\documentclass[11pt]{amsart}

\usepackage[T1]{fontenc}
\usepackage[utf8]{inputenc}
\usepackage{lmodern}
\usepackage{microtype}
\usepackage{amsmath,amsthm,amssymb,mathtools,mathrsfs}
\usepackage{enumitem}
\usepackage[margin=3cm]{geometry}
\usepackage{hyperref}
\hypersetup{colorlinks=true,linkcolor=blue,citecolor=blue,urlcolor=blue}

\newtheorem{theorem}{Theorem}[section]
\newtheorem{proposition}[theorem]{Proposition}
\newtheorem{lemma}[theorem]{Lemma}
\newtheorem{corollary}[theorem]{Corollary}
\theoremstyle{definition}
\newtheorem{remark}[theorem]{Remark}

\numberwithin{equation}{section}

\newcommand{\K}{\mathbb{K}}
\newcommand{\Q}{\mathbb{Q}}
\newcommand{\C}{\mathbb{C}}
\newcommand{\F}{\mathbb{F}}
\newcommand{\Z}{\mathbb{Z}}
\newcommand{\cA}{\mathcal{A}}
\newcommand{\cB}{\mathcal{B}}
\newcommand{\cE}{\mathcal{E}}
\newcommand{\Av}{\mathcal{A}_{\mathrm v}}
\newcommand{\opL}{\mathrm{L}}
\newcommand{\opU}{\mathrm{U}}

\DeclareMathOperator{\End}{End}
\DeclareMathOperator{\GL}{GL}
\DeclareMathOperator{\SL}{SL}
\DeclareMathOperator{\Tr}{Tr}
\DeclareMathOperator{\diag}{diag}
\DeclareMathOperator{\id}{id}
\DeclareMathOperator{\im}{im}
\DeclareMathOperator{\spanop}{span}
\DeclareMathOperator{\rk}{rank}
\DeclareMathOperator{\JMS}{JMS}

\newcommand{\inner}[2]{\left\langle #1,#2\right\rangle}

\title[The Jordan multiplication semigroup of matrix algebras]{The Jordan multiplication semigroup of matrix algebras is the full endomorphism semigroup}
\author{Ilja Gogi\'c, Matija Kazalicki, Mateo Toma\v{s}evi\'c}

\address{I.~Gogi\'c, Department of Mathematics, Faculty of Science, University of Zagreb, Bijeni\v{c}ka 30, 10000 Zagreb, Croatia}
\email{ilja@math.hr}

\address{M.~Kazalicki, Department of Mathematics, Faculty of Science, University of Zagreb, Bijeni\v{c}ka 30, 10000 Zagreb, Croatia}
\email{matija.kazalicki@math.hr}

\address{M.~Toma\v{s}evi\'c, Department of Mathematics, Faculty of Science, University of Zagreb, Bijeni\v{c}ka 30, 10000 Zagreb, Croatia}
\email{mateo.tomasevic@math.hr}

\date{\today}

\begin{document}
	
	\begin{abstract}
		Let $\mathbb{K}$ be a field of characteristic different from $2$, and let $M_n(\mathbb{K})$ be the algebra of all $n\times n$ matrices over $\mathbb{K}$. We consider
		the corresponding special Jordan algebra
		$\mathcal{A}:=M_n(\mathbb{K})^+$ with symmetrized product $A\circ B:=(AB+BA)/2$, and write
		$\mathcal{A}_{\mathrm v}:=M_n(\mathbb{K})$ for the underlying $\mathbb{K}$-vector space of $\mathcal{A}$. For
		$A\in\mathcal{A}$, let $\mathrm{L}_A(X):=A\circ X$ be the multiplication operator. We
		consider the Jordan multiplication semigroup generated by all multiplication
		operators,
		\[
		\mathrm{JMS}(\mathcal{A}):=\langle \mathrm{L}_A:A\in\mathcal{A}\rangle\subseteq \mathrm{End}_{\mathbb{K}}(\mathcal{A}_{\mathrm v}).
		\]
		We prove that $\mathrm{JMS}(\mathcal{A})=\mathrm{End}_{\mathbb{K}}(\mathcal{A}_{\mathrm v})$. Equivalently, every $\mathbb{K}$-linear endomorphism of $\mathcal{A}_{\mathrm v}$ is a composition of
		multiplication operators. The proof is primarily linear-algebraic. The main step is to show that $\mathrm{SL}(\mathcal{A}_{\mathrm v})\subseteq \mathrm{JMS}(\mathcal{A})$ by constructing elementary transvections inside the semigroup. We then prove determinant surjectivity on the unit group of $\mathrm{JMS}(\mathcal{A})$ and combine it with the existence of a singular element of rank $n^2-1$ to obtain the full endomorphism semigroup. In the finite-field case, the determinant-surjectivity step is established via Jacobi-sum estimates.
	\end{abstract}
	
	\keywords{Jordan algebra, Jordan multiplication operator, quadratic operator, transvection, matrix algebra, endomorphism semigroup}
	\subjclass[2020]{17C55, 15A04, 20M20, 15A30}
	
	\maketitle
	
	\section{Introduction}
	Let $\K$ be a field of characteristic different from $2$, and let
	$M_n(\K)$ be the algebra of all $n\times n$ matrices over $\K$. We consider
	the associated \emph{special Jordan algebra}, denoted by $M_n(\K)^+$, obtained
	from $M_n(\K)$ by replacing the associative product with its symmetrization
	\[
	A\circ B:=\frac{1}{2}(AB+BA).
	\]
	For general background on Jordan algebras, see
	\cite{Jacobson1968,McCrimmon2004,Springer1973}.
	
	Throughout the paper, unless explicitly stated otherwise, we write
	\[
	\cA:=M_n(\K)^+,
	\qquad
	\Av:=M_n(\K),
	\qquad
	n\ge 2.
	\]
	We regard $\Av$ as the underlying $\K$-vector space of the Jordan algebra
	$\cA$, and write $\End_{\K}(\Av)$ for the algebra of all $\K$-linear
	endomorphisms of $\Av$. For each $A\in\cA$, we write
	\[
	\opL_A\colon \Av\to \Av,
	\qquad
	\opL_A(X):=A\circ X,
	\]
	for the corresponding \emph{multiplication operator}. We then define the
	\emph{Jordan multiplication semigroup} of $\cA$ to be the subsemigroup of
	$\End_{\K}(\Av)$ generated by all multiplication operators $\opL_A$, namely
	\[
	\JMS(\cA):=\langle \opL_A:A\in \cA\rangle\subseteq \End_{\K}(\Av).
	\]
	Since $\opL_{I_n}=\id_{\Av}$, the semigroup $\JMS(\cA)$ is in fact a monoid.
	
	A natural background for the present paper is the theory of automatic additivity and, more broadly, the rigidity theory of multiplicative maps. In this circle of problems one typically studies multiplicative maps under additional assumptions such as injectivity, surjectivity, or bijectivity, and asks whether multiplicativity already forces additivity, semilinearity, or some standard algebraic form. In the associative setting, this theme appears in work of Rickart, Johnson, Martindale, Jodeit--Lam, \v{S}emrl, and others \cite{Rickart1948,Johnson1958,Martindale1969,JodeitLam1969,Semrl1995,Semrl2008}. Analogous rigidity phenomena also occur in nonassociative settings. For Lie multiplicative maps, one often obtains additivity, or almost additivity in the sense of additivity up to central-valued terms, under suitable hypotheses \cite{Semrl2005,Dolinar2007Upper,Dolinar2007Full,QiHou2011,WangZhaoChen2011}. For Jordan homomorphisms, the classical theory goes back to the foundational work of Jacobson--Rickart, Herstein, and Smiley \cite{JacobsonRickart1950,Herstein1956,Smiley1957}. In the modern preserver-theoretic setting, Jordan maps and Jordan multiplicative maps exhibit analogous automatic additivity and rigidity phenomena \cite{Molnar2002,Lu2002,AnHou2006,JiLiu2009,GogicTomasevic2025-AM,GogicTomasevic2025-LMA}.
	
	The present paper is motivated by this rigidity perspective, but approaches it
	from a different angle. Rather than studying Jordan multiplicative maps
	directly, we investigate the linear semigroup generated by the Jordan
	multiplication operators themselves. The connection with Jordan multiplicative maps is that every such map
	intertwines the corresponding multiplication-operator calculus: a map
	$\phi\colon \cA\to \cA$ satisfies
	\[
	\phi(A\circ B)=\phi(A)\circ \phi(B),
	\qquad\text{for all }A,B\in\cA,
	\]
	if and only if
	\[
	\phi\circ \opL_A=\opL_{\phi(A)}\circ \phi,
	\qquad\text{for all }A\in\cA,
	\]
	as self-maps of the underlying set of $\cA$. Thus every finite composition of multiplication operators intertwines through $\phi$ with the corresponding
	composition formed from the operators $\opL_{\phi(A)}$. In this sense, $\JMS(\cA)$ records a linear part of the multiplicative structure that is
	naturally visible to every Jordan multiplicative map. The problem studied here is to determine how large this linear structure is for
	$\cA=M_n(\K)^+$.
	
	A starting point for the present paper is a recent result of the first and third authors. They showed that every Jordan multiplicative self-map of $M_n(\K)^+$ is either constant or a Jordan ring monomorphism; in particular, every nonconstant Jordan multiplicative self-map is automatically injective and additive \cite{GogicTomasevic2025-AM}. The key input in their argument is a strong nonlinear rigidity property of Jordan multiplication: for every nonzero matrix $X\in M_n(\K)$, the iterated Jordan-multiplication closure generated by $X$, equivalently the orbit $\JMS(\cA)\cdot X$, coincides with the whole algebra $M_n(\K)$ \cite[Proposition~2.2]{GogicTomasevic2025-AM}. In an algebraically closed field of characteristic different from $2$, the corresponding zero-or-all statement was later proved in the setting of mixed Jordan-power closure \cite[Proposition~2.9]{GogicTomasevicMixed2025}.
	
	The problem addressed here is whether, for the full matrix Jordan algebra, this orbit-maximality is the shadow of a stronger linear fact: that the semigroup $\JMS(\cA)$ is already maximal inside $\End_{\K}(\Av)$.
	
	Our main result shows that this is indeed the case.
	
	\begin{theorem}\label{thm:main-all-fields}
		Let $\K$ be a field of characteristic different from $2$, and let $n\ge 1$. Then
		\[
		\JMS(M_n(\K)^+)=\End_{\K}(M_n(\K)).
		\]
	\end{theorem}
	
	This phenomenon has no associative analogue. If $\cB$ is a unital associative $\K$-algebra and $\Lambda_A\colon \cB\to\cB$ denotes left multiplication by $A\in\cB$, that is,
	\[
	\Lambda_A(X):=AX \qquad (X\in\cB),
	\]
	then
	\[
	\Lambda_A\Lambda_B=\Lambda_{AB}, \qquad \text{for all } A,B\in\cB,
	\]
	so the semigroup generated by the left multiplication operators is simply
	\[
	\Lambda(\cB):=\{\Lambda_A:A\in\cB\}.
	\]
	Moreover, if $\Pi_B\colon \cB\to\cB$ denotes right multiplication by $B\in\cB$, that is,
	\[
	\Pi_B(X):=XB
	\qquad (X\in\cB),
	\]
	then every $\Lambda_A$ commutes with every $\Pi_B$. Hence this semigroup can coincide with $\End_{\K}(\cB_{\mathrm v})$ only in the trivial one-dimensional case. Thus the maximality statement of Theorem~\ref{thm:main-all-fields} is genuinely Jordan-theoretic.
	
	It is also worth noting that Theorem~\ref{thm:main-all-fields} is stronger than the statement that $\JMS(\cA)$ linearly spans $\End_{\K}(\Av)$. Indeed, for $A,B\in M_n(\K)$, one has
	\begin{equation}\label{eq:intro-LA-comm}
		2\opL_A=\Lambda_A+\Pi_A,
		\qquad
		4[\opL_A,\opL_B]=\Lambda_{[A,B]}-\Pi_{[A,B]},
	\end{equation}
	where $[A,B]:=AB-BA$ denotes the associative commutator. Since every traceless matrix is an associative commutator \cite{AlbertMuckenhoupt1957}, while $I_n$ supplies the scalar part, it follows from \eqref{eq:intro-LA-comm} that the associative subalgebra of $\End_{\K}(\Av)$ generated by the operators $\opL_A$ contains both $\Lambda_C$ and $\Pi_C$ for every $C\in M_n(\K)$. In particular, for standard matrix units $E_{ij},E_{k\ell}\in M_n(\K)$ one has
	\[
	\Lambda_{E_{ij}}\Pi_{E_{k\ell}}(X)=E_{ij}XE_{k\ell}=X_{jk}E_{i\ell},
	\]
	so these operators form a complete system of matrix units for $\End_{\K}(\Av)$. Therefore the associative subalgebra generated by the operators $\opL_A$ is already all of $\End_{\K}(\Av)$. Since $\JMS(\cA)$ consists exactly of finite products of the operators $\opL_A$ and contains $\id_{\Av}=\opL_{I_n}$, it follows that $\spanop_{\K}\JMS(\cA)=\End_{\K}(\Av)$. The content of Theorem~\ref{thm:main-all-fields} is that one does not need to pass even to linear combinations: finite compositions of the operators $\opL_A$ already yield every endomorphism of $\Av$.
	
	\smallskip
	
	For $n=1$ the statement of Theorem~\ref{thm:main-all-fields} is immediate, since $M_1(\K)^+=\K$ and $\opL_A(X)=AX$. Accordingly, the substance of the paper begins with the case $n\ge 2$. The proof then passes through the following intermediate statement.
	
	\begin{theorem}\label{thm:SL-main}
		Let $\K$ be a field of characteristic different from $2$, and let $n\ge 2$. Then
		\[
		\SL(\Av)\subseteq \JMS(\cA).
		\]
		Equivalently, after choosing any $\K$-basis of $\Av$, the semigroup $\JMS(\cA)$ contains every elementary transvection of $\SL_{n^2}(\K)$ and therefore the whole group $\SL_{n^2}(\K)$.
	\end{theorem}
	
	This theorem is the conceptual core of the paper.  Recall that a
	\emph{transvection} on a finite-dimensional vector space $W$ is a nontrivial
	unipotent linear operator that fixes pointwise a hyperplane and acts
	nontrivially along a one-dimensional direction; see, e.g.,
	\cite{HahnOMeara1989,Humphries1986}. After choosing a basis
	$(b_1,\dots,b_d)$ of $W$, the most basic examples are the \emph{elementary
		transvections}
	\[
	x_{ij}(t):=\id_W+t\,\varepsilon_{ij}
	\qquad (i\neq j,\ t\in\K),
	\]
	where $\varepsilon_{ij}$ is the rank-one operator sending $b_j$ to $b_i$
	and all other basis vectors to $0$. These elementary transvections are the
	standard root subgroups of $\SL(W)$, and a classical Gaussian-elimination
	argument shows that they generate the whole special linear group. Thus the
	proof of Theorem~\ref{thm:SL-main} reduces to showing that $\JMS(\cA)$
	contains all elementary transvections relative to a suitable basis of $\Av$;
	this reduction is carried out in Lemma~\ref{lem:gauss}.
	
	Sections~2 and~3 implement this strategy. Section~2 constructs the basic transvections intrinsically from Jordan multiplication. For each rank-one square-zero matrix $N$, the associated \emph{quadratic operator}
	\[
	\opU_N\colon \Av\to\Av, \qquad \opU_N(X):=NXN
	\]
	is a rank-one nilpotent endomorphism in $\JMS(\cA)$ (Lemma~\ref{lem:UN-properties}), so that
	\[
	u_N(t):=\id_{\Av}+t\opU_N
	\]
	is a transvection whenever $t\neq 0$. The main result of the section, Proposition~\ref{prop:all-root-transvections}, proves that
	\[
	u_N(t)\in \JMS(\cA)
	\qquad (t\in\K).
	\]
	When $\K$ has characteristic $0$, the key input is a concrete operator identity over $\Q$ expressing $\id_{\Av}+\opU_{E_{12}}$ as a finite product of Jordan multiplication operators; this identity is recorded in Appendix~\ref{app:Q-factorization}. Once this special case is available, the argument extends it to arbitrary rank-one square-zero matrices by conjugation and then to arbitrary parameters by a simple reduction.
	
	Section~3 then shows that these Jordan-theoretic transvections $u_N(t)$ suffice to recover the standard elementary transvections  on $\Av$. The proof first treats the $4$-dimensional space $M_2(\K)$ in a suitable basis, where explicit operator identities produce all elementary transvections; these computations are collected in Appendix~\ref{app:M2-computations} and are used in Proposition~\ref{prop:M2}. One then transports this local $2\times2$ construction to arbitrary corners of $M_n(\K)$ and propagates it to the full standard basis by a commutator argument. Once all elementary transvections are known to lie in $\JMS(\cA)$, Lemma~\ref{lem:gauss} yields Theorem~\ref{thm:SL-main}.
	
	Section~4 completes the passage from Theorem~\ref{thm:SL-main} to Theorem~\ref{thm:main-all-fields}. Once $\SL(\Av)\subseteq \JMS(\cA)$ is established, two further ingredients suffice. First, one studies the determinant map on the group of invertible elements of $\JMS(\cA)$ and shows that every element of $\K^\times$ occurs as the determinant of an invertible element of the semigroup. This implies that every invertible operator on $\Av$ belongs to $\JMS(\cA)$. Second, one exhibits a singular element of rank $n^2-1$ inside $\JMS(\cA)$. Since all invertible operators are now available, this corank-one operator can be moved to arbitrary position by left and right composition, yielding all rank-$(n^2-1)$ endomorphisms; composing such operators then gives every singular endomorphism. Thus $\JMS(\cA)=\End_{\K}(\Av)$. The only genuinely field-dependent point in this final step is the determinant-surjectivity argument; in the finite-field case it is established using standard character-sum estimates.
	
	Only two parts of the proof require explicit computation. Appendix~\ref{app:Q-factorization} gives the operator identity used in the characteristic-zero case of Proposition~\ref{prop:all-root-transvections}, and Appendix~\ref{app:M2-computations} records the exact $M_2$ identities used in Proposition~\ref{prop:M2}. These identities were first found with the assistance of OpenAI's ChatGPT and were subsequently verified by exact symbolic computation. Short computer-algebra verification files accompany the paper as supplementary material.
	
	Finally, the characteristic assumption is essential. In characteristic $2$, the normalized Jordan product is unavailable, and even for the unnormalized symmetrized product $A\diamond B:=AB+BA$ one has $\opL_A(X)=AX+XA=[A,X]$. Hence every nonempty composition annihilates $\K I_n$ and has image contained in $\mathfrak{sl}_n(\K)$. In particular, the full endomorphism semigroup cannot be generated in that setting.
	
	\section{Quadratic operators and rank-one transvections}\label{sec:quadratic}
	This section provides the first structural step in the proof of Theorem~\ref{thm:SL-main}. We show that rank-one square-zero matrices give rise to transvections in $\JMS(\cA)$ via the Jordan quadratic operators. These unipotent elements will be used in the next section to generate the elementary transvections of $\SL(\Av)$.
	
	For the remainder of the paper, we assume that $\K$ is a field of
	characteristic different from $2$ and that $n\ge 2$. We also use the symmetric
	nondegenerate trace pairing
	\[
	\inner{X}{Y}:=\Tr(XY)
	\qquad (X,Y\in \Av).
	\]
	For $A\in\cA$, recall that the \emph{Jordan quadratic operator} is
	\[
	\opU_A\colon \Av\to\Av,
	\qquad
	\opU_A:=2\opL_A^2-\opL_{A^2};
	\]
	see, for example, \cite[Chapter~1]{Jacobson1969}. In the special Jordan algebra $M_n(\K)^+$ this is given by
	\[
	\opU_A(X)=AXA
	\qquad (X\in \Av).
	\]
	In particular, if $N^2=0$, then $\opU_N=2\opL_N^2$. Thus $\opU_N\in\JMS(\cA)$
	for every square-zero matrix $N$. The next lemma identifies the rank-one case relevant for transvections.
	
	\begin{lemma}\label{lem:UN-properties}
		Let $N\in M_n(\K)$ satisfy $N^2=0$, and define
		\[
		\opU_N(X):=NXN
		\qquad (X\in \Av).
		\]
		Then:
		\begin{enumerate}[label={\rm(\alph*)}]
			\item One has
			\[
			\opL_N^2=\frac12\opU_N.
			\]
			In particular, $\opU_N\in \JMS(\cA)$.
			
			\item If moreover $\rk(N)=1$, then
			\[
			\opU_N(X)=\inner{N}{X}N
			\qquad (X\in \Av).
			\]
			In particular, $\rk(\opU_N)=1$ and $\opU_N^2=0$.
		\end{enumerate}
	\end{lemma}
	
	\begin{proof}
		(a) For $X\in\Av$,
		\begin{align*}
			\opL_N^2(X)&=\frac14\bigl(N(NX+XN)+(NX+XN)N\bigr)
			=\frac14\bigl(N^2X+2NXN+XN^2\bigr)\\
			&=\frac12NXN.
		\end{align*}
		Thus $\opL_N^2=\frac12\opU_N$. Since $\opL_{2I_n}=2\id_{\Av}$, it follows that
		\[
		\opU_N=\opL_{2I_n}\opL_N^2\in \JMS(\cA).
		\]
		
		\smallskip
		
		(b) Assume in addition that $\rk(N)=1$. Write
		\[
		N=uv^t
		\]
		for some nonzero vectors $u,v\in\K^n$, where, as usual, vectors are treated as
		columns and $v^t$ denotes the transpose of $v$. Then
		\[
		N^2=(v^tu)N,
		\]
		so $N^2=0$ is equivalent to $v^tu=0$. For $X\in\Av$,
		\[
		\opU_N(X)=uv^tXuv^t=(v^tXu)N.
		\]
		Also,
		\[
		\Tr(NX)=\Tr(uv^tX)=\Tr(v^tXu)=v^tXu.
		\]
		Hence
		\[
		\opU_N(X)=\inner{N}{X}N.
		\]
		It follows that $\operatorname{im}(\opU_N)=\K N$, so $\rk(\opU_N)=1$. Finally,
		\[
		\opU_N^2(X)=\inner{N}{X}\opU_N(N)=\inner{N}{X}N^3=0.
		\]
	\end{proof}
	
	Lemma~\ref{lem:UN-properties} shows that rank-one square-zero matrices give rise to rank-one nilpotent endomorphisms of $\Av$. The associated unipotent operators are therefore transvections. 
	
	\begin{remark}\label{rem:root-transvection} 
		As already mentioned in the introduction, a transvection on a finite-dimensional
		vector space $W$ is a nontrivial unipotent operator fixing pointwise a
		hyperplane and acting nontrivially along a one-dimensional direction.
		Equivalently, it is a nontrivial unipotent linear map
		\[
		W\to W,\qquad x\mapsto x+\varphi(x)v,
		\]
		where $v\in W$, $\varphi\in W^*$ is a linear functional, and
		$v\in\ker\varphi$.
		
		If $N\in M_n(\K)$ has rank $1$ and $N^2=0$, then for every $t\in\K$ the operator
		\[
		u_N(t):=\id_{\Av}+t\opU_N
		\]
		satisfies
		\[
		u_N(t)^{-1}=u_N(-t).
		\]
		For $t\neq 0$, Lemma~\ref{lem:UN-properties} gives
		\[
		u_N(t)(X)=X+t\,\inner{N}{X}N,
		\]
		so $u_N(t)$ is a transvection on $\Av$ with direction $\K N$ and fixed hyperplane
		\[
		\ker\bigl(X\mapsto \inner{N}{X}\bigr)=\ker(\opU_N).
		\]
		For $t=0$, one has $u_N(0)=\id_{\Av}$.
	\end{remark}
	
	To pass from the special matrix $E_{12}$ to an arbitrary rank-one square-zero matrix, we will use conjugation. The following lemma shows that both $\opL_A$ and $\opU_A$ behave naturally under this operation.
	\begin{lemma}\label{lem:conjugation-covariance}
		For $S\in\GL_n(\K)$, let
		\[
		\Phi_S(X):=SXS^{-1}
		\qquad (X\in\Av).
		\]
		Then for every $A\in M_n(\K)$ one has
		\[
		\Phi_S\opL_A\Phi_S^{-1}=\opL_{SAS^{-1}}
		\qquad\text{and}\qquad
		\Phi_S\opU_A\Phi_S^{-1}=\opU_{SAS^{-1}}.
		\]
	\end{lemma}
	
	\begin{proof}
		For fixed $X\in\Av$ we have
		\begin{align*}
			\Phi_S\opL_A\Phi_S^{-1}(X)
			&=
			S\left(A\circ(S^{-1}XS)\right)S^{-1}
			=
			\frac{SAS^{-1}X+XSAS^{-1}}{2}\\
			&=
			\opL_{SAS^{-1}}(X).
		\end{align*}
		Similarly,
		\begin{align*}
			\Phi_S\opU_A\Phi_S^{-1}(X)
			&=
			S\bigl(A(S^{-1}XS)A\bigr)S^{-1}
			=
			(SAS^{-1})X(SAS^{-1})\\
			&= \opU_{SAS^{-1}}(X).
		\end{align*}
	\end{proof}
	
	The next lemma shows that it is enough to construct the two special operators $\id_{\Av}\pm\opU_M$. Once these are available for every rank-one square-zero matrix $M$, the full family $\id_{\Av}+t\opU_N$ follows formally.
	\begin{lemma}\label{lem:pm1}
		Assume that for every rank-one square-zero matrix $M\in M_n(\K)$ one has
		\[
		\id_{\Av}\pm\opU_M\in\JMS(\cA).
		\]
		Then for every rank-one square-zero matrix $N\in M_n(\K)$ and every $t\in\K$,
		\[
		\id_{\Av}+t\opU_N\in\JMS(\cA).
		\]
	\end{lemma}
	
	\begin{proof}
		Fix $N$ and $t$. Since $\operatorname{char}\K\ne 2$, we may write
		\[
		t=x^2-y^2,
		\qquad \text{where} \qquad  
		x:=\frac{t+1}{2},
		\qquad
		y:=\frac{t-1}{2}.
		\]
		Let $z\in\K$. If $z=0$, then $\id_{\Av}\pm z^2\opU_N=\id_{\Av}\in\JMS(\cA)$. If $z\ne 0$, then $zN$ is again rank one and square-zero, so $\id_{\Av}\pm\opU_{zN}\in\JMS(\cA)$. Since $\opU_{zN}=z^2\opU_N$, we have in all cases
		\[
		\id_{\Av}\pm z^2\opU_N\in\JMS(\cA), \qquad \text{for all } z\in\K.
		\]
		Applying this with $z=x$ and $z=y$, and using $\opU_N^2=0$, we get
		\[
		(\id_{\Av}+x^2\opU_N)(\id_{\Av}-y^2\opU_N)
		=\id_{\Av}+(x^2-y^2)\opU_N
		=\id_{\Av}+t\opU_N\in\JMS(\cA).
		\]
	\end{proof}
	
	It remains to construct the two operators $\id_{\Av}\pm\opU_M$ for rank-one square-zero matrices $M$. We begin with the basic rank-one square-zero matrix $M=E_{12}$ over $\Q$.
	\begin{lemma}\label{lem:E12-over-Q}
		For every $n\ge 2$, we have
		\[
		\id_{\Av}\pm \opU_{E_{12}}\in\JMS(M_n(\Q)^+).
		\]
	\end{lemma}
	
	\begin{proof}
		Appendix~\ref{app:Q-factorization} gives an explicit factorization of
		$\id_{\Av}+\opU_{E_{12}}$ on $M_2(\Q)$, and then extends it to
		arbitrary $n\ge 2$ by padding. The minus sign is obtained from the
		elementary identity
		\[
		\id_{\Av}-\opU_{E_{12}}
		=
		\bigl(\opL_{I+E_{12}}\opL_{I-E_{12}}\bigr)^2.
		\]
	\end{proof}
	We can now assemble the preceding ingredients into the main result of the section.
	
	\begin{proposition}\label{prop:all-root-transvections}
		Let $\K$ be a field of characteristic different from $2$, and let $N\in M_n(\K)$ satisfy $\rk(N)=1$ and $N^2=0$. Then for all $t\in\K$,
		\[
		u_N(t):=\id_{\Av}+t\opU_N\in\JMS(\cA).
		\]
	\end{proposition}
	\begin{proof}
		By Lemma~\ref{lem:pm1}, it is enough to prove that $\id_{\Av}\pm\opU_M\in\JMS(\cA)$ for every rank-one square-zero matrix $M$. We treat separately the cases $\operatorname{char}\K>0$ and $\operatorname{char}\K=0$. In positive characteristic, the operators $\id_{\Av}\pm\opU_M$ are obtained directly from a simple explicit element of $\JMS(\cA)$. In characteristic $0$, one first establishes the corresponding identities for the matrix $E_{12}$ over $\Q$ and then transfers them to an arbitrary rank-one square-zero matrix by conjugation.
		
		Assume first that $\operatorname{char}\K=p>2$. By Lemma~\ref{lem:UN-properties},
		\[
		\opL_{I_n+M}\opL_{I_n-M}=\id_{\Av}-\opL_M^2=\id_{\Av}-\frac12\opU_M\in\JMS(\cA).
		\]
		Since $\opU_M^2=0$, repeated multiplication gives
		\[
		\bigl(\id_{\Av}-\tfrac12\opU_M\bigr)^r=\id_{\Av}-\frac r2\,\opU_M
		\qquad (r\ge 0).
		\]
		Since $-1/2\neq 0$ in $\F_p$, it generates the additive group of the prime field $\F_p$. Hence $\id_{\Av}+s\opU_M\in\JMS(\cA)$ for every $s\in\F_p$, in particular for $s=\pm 1$.
		
		Now assume that $\operatorname{char}\K=0$. Then $\Q$ identifies canonically
		with the prime subfield of $\K$. By Lemma~\ref{lem:E12-over-Q}, there exist
		matrices $A_1,\dots,A_m\in M_n(\Q)$ such that
		\[
		\opL_{A_1}\cdots \opL_{A_m}
		=
		\id_{\Av}\pm\opU_{E_{12}}
		\]
		as $\Q$-linear endomorphisms of $M_n(\Q)$. We may therefore regard the same
		matrices $A_1,\dots,A_m$ as elements of $M_n(\K)$. Since the displayed identity
		is an exact equality of operator matrices with entries in $\Q$, the same
		equality holds over $\K$. Therefore
		\[
		\id_{\Av}\pm\opU_{E_{12}}\in\JMS(\cA).
		\]
		
		It remains to pass from the matrix $E_{12}$ to an arbitrary rank-one square-zero matrix. The following argument uses Lemma~\ref{lem:conjugation-covariance} to transport the identities already established for $E_{12}$ to the general case. Indeed, let $M$ be any rank-one square-zero matrix. Choose $0\neq b_1\in\im(M)$.
		Since $\rk(M)=1$, we have $\im(M)=\K b_1$. Because $M\neq 0$, there exists
		$b_2\in\K^n$ such that $Mb_2=b_1$. Also $M^2=0$ implies $\im(M)\subseteq\ker(M)$,
		so $b_1\in\ker(M)$. As $\dim\ker(M)=n-1$, we may extend $b_1$ to a basis $b_1,b_3,\dots,b_n$
		of $\ker(M)$. Then $(b_1,b_2,b_3,\dots,b_n)$ is a basis of $\K^n$, and relative to this basis the linear map $M$ sends $b_2$ to $b_1$ and annihilates $b_1,b_3,\dots,b_n$. Hence the matrix of $M$
		in this basis is exactly $E_{12}$. Thus there exists $S\in\GL_n(\K)$ such that
		\[
		M=SE_{12}S^{-1}.
		\] 
		By Lemma~\ref{lem:conjugation-covariance},
		\[
		\Phi_S\opU_{E_{12}}\Phi_S^{-1}=\opU_M.
		\]
		By Lemma~\ref{lem:E12-over-Q}, for each choice of sign there exist
		$A_1,\ldots,A_m\in M_n(\Q)^+$ such that
		\[
		T:=\opL_{A_1}\cdots \opL_{A_m}
		=
		\id_{\Av}\pm\opU_{E_{12}}.
		\]
		Then
		\[
		\Phi_S T \Phi_S^{-1}
		=
		\opL_{SA_1S^{-1}}\cdots \opL_{SA_mS^{-1}}\in\JMS(\cA),
		\]
		while
		\[
		\Phi_S T \Phi_S^{-1}
		=
		\Phi_S(\id_{\Av}\pm\opU_{E_{12}})\Phi_S^{-1}
		=
		\id_{\Av}\pm\opU_M.
		\]
		Therefore $\id_{\Av}\pm\opU_M\in\JMS(\cA)$, so Lemma~\ref{lem:pm1} completes the proof.
	\end{proof}
	
	\begin{remark}\label{rem:algclosed-root-transvections}
		If $\K$ is algebraically closed, the proof of
		Proposition~\ref{prop:all-root-transvections} becomes much simpler. Indeed, if
		$N^2=0$, then for every $s\in\K$ one has $\opL_{I_n+sN}=\id_{\Av}+s\opL_N$, and hence
		\[
		\opL_{I_n+sN}\opL_{I_n-sN}
		=
		(\id_{\Av}+s\opL_N)(\id_{\Av}-s\opL_N)
		=
		\id_{\Av}-s^2\opL_N^2
		=
		\id_{\Av}-\frac{s^2}{2}\opU_N.
		\]
		Given $t\in\K$, choose $s\in\K$ with $s^2=-2t$. Then
		\[
		u_N(t)=\id_{\Av}+t\opU_N
		=
		\opL_{I_n+sN}\opL_{I_n-sN}\in\JMS(\cA).
		\]
	\end{remark}
	
	\section{From root transvections to \texorpdfstring{$\SL(\Av)$}{SL(A\_v)}}\label{sec:SL}
	
	In this section we prove Theorem~\ref{thm:SL-main}. We first work on the $4$-dimensional space $M_2(\K)$, equipped with a suitable basis, and show that the transvections constructed in Section~2 already yield all elementary transvections there. We then transfer this construction to arbitrary $2\times2$ corners of $M_n(\K)$ and use a simple propagation argument to reach the full standard basis of $\Av$. It follows that $\JMS(\cA)$ contains all elementary transvections of $\SL(\Av)$, and hence $\SL(\Av)$ itself.
	
	For every rank-one square-zero matrix $R$ we set
	\[
	g_R:=\opL_{I_n+R}.
	\]
	The operators $g_R$ and their inverses will be used to conjugate the transvections constructed in the previous section. We begin by recording that these inverses still lie in the semigroup.
	
	For invertible operators $P,Q$ we write
	\[
	[P,Q]_{\mathrm{grp}}:=PQP^{-1}Q^{-1}
	\]
	for their group commutator.
	
	\begin{lemma}\label{lem:gR-inverse}
		Let $R\in M_n(\K)$ satisfy $\rk(R)=1$ and $R^2=0$. Then
		\[
		g_R^{-1}=\id_{\Av}-\opL_R+\opL_R^2=\Bigl(\id_{\Av}+\frac12\opU_R\Bigr)\opL_{I_n-R}\in\JMS(\cA).
		\]
	\end{lemma}
	
	\begin{proof}
		By Lemma~\ref{lem:UN-properties}, $\opL_R^2=\frac12\opU_R$. Also,
		\[
		g_R=\opL_{I_n+R}=\id_{\Av}+\opL_R.
		\]
		Moreover,
		\[
		\opL_R^3(X)=\frac12\opL_R(\opU_R(X))=\frac14\bigl(R(RXR)+(RXR)R\bigr)=0
		\]
		for every $X\in\Av$, because $R^2=0$. Therefore
		\[
		(\id_{\Av}+\opL_R)(\id_{\Av}-\opL_R+\opL_R^2)=\id_{\Av}+\opL_R^3=\id_{\Av},
		\]
		so $g_R^{-1}=\id_{\Av}-\opL_R+\opL_R^2$. Since $\opL_{I_n-R}=\id_{\Av}-\opL_R$, it remains to check that
		\[
		\left(\id_{\Av}+\frac12\opU_R\right)\opL_{I_n-R}=\id_{\Av}-\opL_R+\frac12\opU_R.
		\]
		But
		\[
		\opU_R\opL_R(X)=R\left(\frac{RX+XR}{2}\right)R=\frac12(R^2XR+RXR^2)=0,
		\]
		so indeed 
		\[
		(\id_{\Av}+\frac12\opU_R)\opL_{I_n-R}=\id_{\Av}-\opL_R+\frac12\opU_R.
		\]
		Since $u_R(1/2)=\id_{\Av}+\frac12\opU_R$ belongs to $\JMS(\cA)$ by Proposition~\ref{prop:all-root-transvections}, the inverse lies in $\JMS(\cA)$.
	\end{proof}
	
	We first work on the $4$-dimensional space $M_2(\K)$ in a suitable basis. The next proposition shows that the transvections constructed in Section~2 already suffice there to produce all elementary transvections.
	
	\begin{proposition}\label{prop:M2}
		Define an ordered basis $\cE=(B_1,B_2,B_3,B_4)$ of $M_2(\K)$ by
		\[
		B_1:=E_{21}+\frac12I_2,
		\qquad
		B_2:=E_{12},
		\qquad
		B_3:=I_2+E_{12},
		\qquad
		B_4:=E_{11}-E_{22}.
		\]
		For $1\le r\neq s\le 4$ and $t\in\K$, let $x_{rs}(t)$ be the elementary transvection relative to $\cE$:
		\[
		x_{rs}(t)(B_s)=B_s+tB_r,
		\qquad
		x_{rs}(t)(B_k)=B_k\quad (k\neq s).
		\]
		Then
		\[
		x_{rs}(t)\in\JMS(M_2(\K)^+)
		\qquad (1\le r\neq s\le 4,\ t\in\K).
		\]
	\end{proposition}
	
	\begin{remark}\label{rem:epsilon-notation}
		Let $W$ be a finite-dimensional vector space with ordered basis
		$\cE=(b_1,\dots,b_d)$. For basis vectors $u,v\in\cE$, let
		\[
		\varepsilon_{u,v}\in\End_{\K}(W)
		\]
		denote the rank-one operator sending $v$ to $u$ and every other basis vector of
		$\cE$ to $0$. In particular, for $1\le i,j\le d$ we write
		\[
		\varepsilon_{ij}:=\varepsilon_{b_i,b_j}.
		\]
		When $W=M_2(\K)$ with the ordered basis $\cE=(B_1,B_2,B_3,B_4)$ from
		Proposition~\ref{prop:M2}, the symbols $\varepsilon_{rs}$ refer to these matrix
		units and are distinguished from the standard matrix units
		$E_{ij}\in M_2(\K)$.
	\end{remark}
	
	\begin{proof}[Proof of Proposition~\ref{prop:M2}]
		Using the notation of Remark~\ref{rem:epsilon-notation}, for $1\le r\ne s\le 4$ we have
		\[
		x_{rs}(t)=\id_{M_2(\K)}+t\varepsilon_{rs}.
		\]
		The basis $\cE$ is chosen so that $\opU_{E_{12}}$ acts as a single matrix unit and so that a family of conjugates $g_Pu_Q(t)g_P^{-1}$, with $P,Q \in M_2(\K)$ rank-one square-zero matrices, considered below has sparse matrices in this basis. The exact formulas used below are recorded in Lemma~\ref{lem:M2-formulas} of Appendix~\ref{app:M2-computations}. The change-of-basis matrix from the standard basis $(E_{11},E_{12},E_{21},E_{22})$ to $\cE$ has determinant $2$, so $\cE$ is indeed a basis of $M_2(\K)$ because $\operatorname{char}\K\neq 2$.
		
		The matrices
		\[
		R_+:=\begin{bmatrix}1&1\\ -1&-1\end{bmatrix},
		\qquad
		R_-:=\begin{bmatrix}-1&1\\ -1&1\end{bmatrix},
		\qquad
		R:=\begin{bmatrix}-2&-1\\ 4&2\end{bmatrix},
		\qquad
		S:=\begin{bmatrix}-1&-1\\ 1&1\end{bmatrix}
		\]
		all have rank $1$ and square $0$, and so does $2E_{21}$. Hence every operator of the form $u_N(t)$ and $g_Pu_Q(t)g_P^{-1}$ appearing below lies in $\JMS(M_2(\K)^+)$ by Proposition~\ref{prop:all-root-transvections} and Lemma~\ref{lem:gR-inverse}. For the reader's convenience, we indicate the order in which the twelve
		elementary transvections are obtained:
		\[
		x_{12},x_{21};\qquad
		x_{13},x_{14};\qquad
		x_{23},x_{24};\qquad
		x_{32},x_{31};\qquad
		x_{42},x_{41},x_{43},x_{34}.
		\]
		We now construct these families in turn.
		
		The formulas from Lemma~\ref{lem:M2-formulas} give
		\[
		u_{E_{12}}(t)=\id_{M_2(\K)}+t\varepsilon_{21}=x_{21}(t),
		\qquad
		g_{E_{12}}u_{E_{21}}(t)g_{E_{12}}^{-1}=\id_{M_2(\K)}+t\varepsilon_{12}=x_{12}(t).
		\]
		Hence $x_{12}(t),x_{21}(t)\in\JMS(M_2(\K)^+)$. Replacing $t$ by $-t$ gives their inverses, since $x_{rs}(t)^{-1}=x_{rs}(-t)$ for every elementary transvection.
		
		Next set
		\[
		r_{\pm}(t):=g_{R_{\pm}}u_{E_{21}}(t)g_{R_{\pm}}^{-1}.
		\]
		Lemma~\ref{lem:M2-formulas} gives
		\[
		r_+(t)=\id_{M_2(\K)}+t\left(\frac12\varepsilon_{12}-\frac12\varepsilon_{13}+\varepsilon_{14}\right),
		\qquad
		r_-(t)=\id_{M_2(\K)}+t\left(\frac12\varepsilon_{12}-\frac12\varepsilon_{13}-\varepsilon_{14}\right).
		\]
		Since the matrix units $\varepsilon_{12},\varepsilon_{13},\varepsilon_{14}$ pairwise multiply to zero,
		\[
		r_+\!\left(\frac t2\right)r_-\!\left(-\frac t2\right)=\id_{M_2(\K)}+t\varepsilon_{14}=x_{14}(t),
		\]
		and
		\[
		r_+(-t)r_-(-t)x_{12}(t)=\id_{M_2(\K)}+t\varepsilon_{13}=x_{13}(t).
		\]
		Thus $x_{13}(t),x_{14}(t)\in\JMS(M_2(\K)^+)$. For all distinct indices
		$i,j,k$ and all $\alpha,\beta\in\K$, the standard Steinberg relation
		\[
		[x_{ij}(\alpha),x_{jk}(\beta)]_{\mathrm{grp}}=x_{ik}(\alpha\beta)
		\]
		holds; see, for example, \cite[Chapter~I]{HahnOMeara1989}. Since both a transvection and its inverse belong to the semigroup, each displayed commutator below is a finite product of already known semigroup elements and therefore belongs to the semigroup:
		\[
		x_{23}(t)=[x_{21}(1),x_{13}(t)]_{\mathrm{grp}},
		\qquad
		x_{24}(t)=[x_{21}(1),x_{14}(t)]_{\mathrm{grp}}.
		\]
		Hence $x_{23}(t),x_{24}(t)\in\JMS(M_2(\K)^+)$. \smallskip
		
		Lemma~\ref{lem:M2-formulas} also gives
		\[
		g_{2E_{21}}u_{E_{12}}(s)g_{2E_{21}}^{-1}
		=
		\id_{M_2(\K)}+2s\,\varepsilon_{32},
		\]
		so
		\[
		x_{32}(t)=g_{2E_{21}}u_{E_{12}}\left(\frac{t}{2}\right)g_{2E_{21}}^{-1}\in\JMS(M_2(\K)^+).
		\]
		The displayed group commutator
		\[
		x_{31}(t)=[x_{32}(t),x_{21}(1)]_{\mathrm{grp}}
		\]
		is again a finite product of already known semigroup elements, so $x_{31}(t)\in\JMS(M_2(\K)^+)$. Finally let
		\[
		v(t):=g_Ru_S(t)g_R^{-1}.
		\]
		Lemma~\ref{lem:M2-formulas} gives
		\[
		v(t)=\id_{M_2(\K)}+t(-\varepsilon_{12}+\varepsilon_{32}+\varepsilon_{42}).
		\]
		Again the relevant matrix units pairwise multiply to zero, so
		\[
		v(t)=x_{12}(-t)x_{32}(t)x_{42}(t).
		\]
		Therefore
		\[
		x_{42}(t)=x_{12}(t)x_{32}(-t)v(t)\in\JMS(M_2(\K)^+).
		\]
		Finally, each of the displayed group commutators
		\[
		x_{41}(t)=[x_{42}(t),x_{21}(1)]_{\mathrm{grp}},
		\qquad
		x_{43}(t)=[x_{42}(t),x_{23}(1)]_{\mathrm{grp}},
		\qquad
		x_{34}(t)=[x_{32}(1),x_{24}(t)]_{\mathrm{grp}}
		\]
		is a finite product of already known semigroup elements. Hence these three operators also lie in the semigroup. At this point every ordered pair $(r,s)$ with $r\neq s$ has been realized.
	\end{proof}
	
	Once all elementary transvections are available in a given basis, the passage to the whole special linear group is standard. We record this in the next lemma for later use.
	
	\begin{lemma}\label{lem:gauss}
		Let $W$ be a finite-dimensional vector space over $\K$ with ordered basis $\cE=(b_1,\dots,b_d)$, where $d\ge 2$, and let $S\subseteq\End_{\K}(W)$ be a semigroup containing every elementary transvection
		\[
		x_{ij}(t)=\id_W+t\,\varepsilon_{ij}
		\qquad (1\le i\ne j\le d,\ t\in\K)
		\]
		relative to $\cE$. Then $\SL(W)\subseteq S$.
	\end{lemma}
	
	\begin{proof}
		This is standard; compare \cite[Chapter~I]{HahnOMeara1989} and \cite{Humphries1986}. Since the lemma will be used repeatedly, we record a short self-contained proof.
		
		Identify $\End_{\K}(W)$ with $M_d(\K)$ via the basis $\cE$. The row and column operations below are carried out in the ambient group $\GL_d(\K)$. We use them only to derive a factorization
		\[
		PGQ=\Delta
		\]
		with $P$ and $Q$ products of elementary transvections and $\Delta$ diagonal; membership in $S$ will then follow because the relevant elementary transvections and their inverses lie in $S$.
		
		Let $G=(g_{rs})\in\GL_d(\K)$; after each row or column operation, we again denote the resulting matrix by 
		$G=(g_{rs})$. By composing on the left with at most one elementary transvection, we may suppose that the $(1,1)$-entry of $G$ is nonzero: if $g_{11}=0$, choose $r>1$ with $g_{r1}\ne 0$ and then choose $t\in\K$ such that $g_{11}+tg_{r1}\ne 0$; replacing $G$ by $x_{1r}(t)G$ achieves this.
		
		For each $r>1$, composing on the left with
		\[
		x_{r1}\!\left(-\frac{g_{r1}}{g_{11}}\right)
		\]
		clears the entry in position $(r,1)$; for each $s>1$, composing on the right with
		\[
		x_{1s}\!\left(-\frac{g_{1s}}{g_{11}}\right)
		\]
		clears the entry in position $(1,s)$. Iterating this argument reduces any invertible matrix to diagonal form by left and right multiplication with elementary transvections. Formally, one argues by induction on $d$: after clearing the first row and first column outside the $(1,1)$-entry, the resulting matrix is block diagonal with nonzero $(1,1)$-entry, and hence its lower-right $(d-1)\times(d-1)$ block is invertible. The same procedure therefore applies to that block.
		
		Now let $G\in\SL_d(\K)$. Then there exist products $P,Q$ of elementary transvections such that
		\[
		\Delta:=PGQ=\diag(a_1,\dots,a_d)
		\]
		with $a_1\cdots a_d=1$. For $1\le i<d$ and $a\in\K^\times$, let $H_i(a)$ be the diagonal matrix with entries $a$ and $a^{-1}$ in positions $i$ and $i+1$, and $1$ elsewhere. Then
		\[
		\Delta=H_1(a_1)H_2(a_1a_2)\cdots H_{d-1}(a_1\cdots a_{d-1}).
		\]
		It therefore suffices to write each $H_i(a)$ as a product of elementary transvections. This reduces to the standard $2\times 2$ identity
		\[
		\begin{bmatrix} a & 0 \\ 0 & a^{-1}\end{bmatrix}
		=
		\begin{bmatrix}1&a-1\\0&1\end{bmatrix}
		\begin{bmatrix}1&0\\1&1\end{bmatrix}
		\begin{bmatrix}1&a^{-1}-1\\0&1\end{bmatrix}
		\begin{bmatrix}1&0\\-a&1\end{bmatrix},
		\]
		embedded into the $(i,i+1)$ block. Hence $\Delta\in S$.
		
		Since $x_{ij}(t)^{-1}=x_{ij}(-t)\in S$, the products $P^{-1}$ and $Q^{-1}$ also belong to $S$. Therefore
		\[
		G=P^{-1}\Delta Q^{-1}\in S.
		\]
	\end{proof}
	
	We now embed the $M_2(\K)$ calculation into arbitrary $2\times2$ corners
	of $M_n(\K)$. Fix $1\le i<j\le n$, and put
	\[
	C_{ij}:=\spanop\{E_{ii},E_{ij},E_{ji},E_{jj}\}\subseteq \Av,
	\]
	and
	\[
	Z_{ij}:=\{X\in\Av:X_{ii}=X_{ij}=X_{ji}=X_{jj}=0\}.
	\]
	Thus
	\[
	\Av=C_{ij}\oplus Z_{ij}.
	\]
	Let
	\[
	\iota_{ij}\colon M_2(\K)\longrightarrow C_{ij}
	\]
	denote the natural corner embedding.
	
	The key observation is that the operators used in the proof of
	Proposition~\ref{prop:M2} can be placed in any chosen $2\times2$ corner.
	Indeed, if $N\in M_2(\K)$ is rank one and square-zero, then
	\[
	\opU_{\iota_{ij}(N)}(X)
	=
	\iota_{ij}(N)X\iota_{ij}(N)
	\]
	is obtained by applying $\opU_N$ to the $2\times2$ principal block of $X$
	with rows and columns $i,j$, and then embedding the result back into
	$C_{ij}$. Hence $u_{\iota_{ij}(N)}(t)$ restricts to $u_N(t)$ on
	$C_{ij}$ and fixes $Z_{ij}$ pointwise.
	
	Similarly, the operators $g_{\iota_{ij}(R)}$ and
	$g_{\iota_{ij}(R)}^{-1}$ preserve both $C_{ij}$ and $Z_{ij}$, and their
	restrictions to $C_{ij}$ are the corresponding operators $g_R$ and
	$g_R^{-1}$ on $M_2(\K)$. Therefore
	$
	g_{\iota_{ij}(R)}u_{\iota_{ij}(S)}(t)g_{\iota_{ij}(R)}^{-1}
	$
	restricts to
	$
	g_Ru_S(t)g_R^{-1}
	$
	on $C_{ij}$. Moreover, since $g_{\iota_{ij}(R)}^{-1}$ preserves
	$Z_{ij}$ and $u_{\iota_{ij}(S)}(t)$ fixes $Z_{ij}$ pointwise, this
	conjugate also fixes $Z_{ij}$ pointwise.
	
	For two distinct vectors $u,v$ in the standard matrix-unit basis
	\[
	\cE:=\{E_{ij}:1\le i,j\le n\}
	\]
	of $\Av$, and for $t\in\K$, let $\tau_{u,v}(t)$ denote the elementary
	transvection relative to $\cE$, sending $v$ to $v+tu$ and fixing all
	other basis vectors.
	
	\begin{proposition}\label{prop:local-SL4}
		For every $1\le i<j\le n$, the semigroup $\JMS(\cA)$ contains every
		operator $g\in\GL(\Av)$ such that
		\[
		g(C_{ij})=C_{ij},\qquad g|_{Z_{ij}}=\id_{Z_{ij}},
		\qquad \det(g|_{C_{ij}})=1.
		\]
		In particular, if $u$ and $v$ are distinct elements of
		\[
		\{E_{ii},E_{ij},E_{ji},E_{jj}\},
		\]
		then
		\[
		\tau_{u,v}(t)\in\JMS(\cA)
		\qquad(t\in\K).
		\]
	\end{proposition}
	
	\begin{proof}
		Let $\cE=(B_1,B_2,B_3,B_4)$ be the ordered basis of $M_2(\K)$ from
		Proposition~\ref{prop:M2}, and let
		\[
		\widetilde{\cE}_{ij}
		:=
		\bigl(\iota_{ij}(B_1),\iota_{ij}(B_2),\iota_{ij}(B_3),\iota_{ij}(B_4)\bigr)
		\]
		be the transported ordered basis of $C_{ij}$.
		
		The proof of Proposition~\ref{prop:M2} shows that every elementary
		transvection of $M_2(\K)$ relative to $\cE$ is a product of operators of the
		form
		\[
		u_N(t)
		\qquad\text{and}\qquad
		g_Ru_S(t)g_R^{-1},
		\]
		where $N,R,S\in M_2(\K)$ are rank-one square-zero matrices. Replacing each
		matrix $N,R,S$ in such a factorization by its corner embedding
		$\iota_{ij}(N),\iota_{ij}(R),\iota_{ij}(S)$, the preceding corner-embedding
		discussion shows that the resulting operator belongs to $\JMS(\cA)$, restricts
		to the corresponding elementary transvection on $C_{ij}$ relative to the basis
		$\widetilde{\cE}_{ij}$, and fixes $Z_{ij}$ pointwise.
		
		Hence $\JMS(\cA)$ contains all elementary transvections of $C_{ij}$ relative
		to $\widetilde{\cE}_{ij}$, extended by the identity on $Z_{ij}$. By
		Lemma~\ref{lem:gauss}, these transvections generate $\SL(C_{ij})$. Therefore
		$\JMS(\cA)$ contains every determinant-one operator supported on $C_{ij}$ and
		equal to the identity on $Z_{ij}$.
		
		In particular, this includes the elementary transvections between the four
		standard matrix units $E_{ii},E_{ij},E_{ji},E_{jj}$ spanning $C_{ij}$.
	\end{proof}

	Graph-theoretic arguments for groups generated by transvections are standard;
	see, for instance, Humphries' work on generation of special linear groups by
	transvections \cite{Humphries1986}. We shall need only the following elementary
	path-propagation form of this idea.
	
	At this point we know how to realize transvections between basis vectors lying in a common $2\times2$ corner. The next lemma shows how such local transvections can be propagated through a connected graph of basis vectors.
	
	\begin{lemma}\label{lem:path}
		Let $W$ be a finite-dimensional vector space with basis $\cE$, and let $S\subseteq\End_{\K}(W)$ be a semigroup. Suppose there is a connected graph $\Gamma$ on the vertex set $\cE$ such that for every edge $\{u,v\}$ of $\Gamma$ and every $t\in\K$, both transvections $\tau_{u,v}(t)$ and $\tau_{v,u}(t)$ belong to $S$. Then $S$ contains every transvection $\tau_{u,v}(t)$ with distinct $u,v\in\cE$.
	\end{lemma}
	
	\begin{proof}
		Fix distinct $u,v\in\cE$, and choose a simple path
		\[
		u=v_0,v_1,\dots,v_m=v
		\]
		in $\Gamma$. We prove by induction on $r$ that
		\[
		\tau_{v_0,v_r}(t)\in S
		\qquad (1\le r\le m,\ t\in\K).
		\]
		For $r=1$ this is immediate from the hypothesis.
		
		Assume $r\ge2$ and that $\tau_{v_0,v_{r-1}}(t)\in S$ for all $t\in\K$. Set
		\[
		\phi:=\tau_{v_0,v_{r-1}}(1),\qquad \psi:=\tau_{v_{r-1},v_r}(t).
		\]
		Then $\phi,\psi\in S$, and also $\phi^{-1},\psi^{-1}\in S$ since $\tau_{a,b}(s)^{-1}=\tau_{a,b}(-s)$. Hence $[\phi,\psi]_{\mathrm{grp}}\in S$.
		
		Using the notation of Remark~\ref{rem:epsilon-notation}, write
		\[
		\phi=\id_W+\varepsilon_{v_0,v_{r-1}},
		\qquad
		\psi=\id_W+t\,\varepsilon_{v_{r-1},v_r}.
		\]
		Then
		\[
		\varepsilon_{v_0,v_{r-1}}^2=\varepsilon_{v_{r-1},v_r}^2=0,
		\qquad
		\varepsilon_{v_0,v_{r-1}}\varepsilon_{v_{r-1},v_r}
		=\varepsilon_{v_0,v_r}.
		\]
		Since the path is simple, $v_r\neq v_0$, and therefore
		\[
		\varepsilon_{v_{r-1},v_r}\varepsilon_{v_0,v_{r-1}}=0.
		\]
		Thus
		\[
		[\phi,\psi]_{\mathrm{grp}}
		=
		\id_W+t\,\varepsilon_{v_0,v_r}
		=
		\tau_{v_0,v_r}(t).
		\]
		This proves the induction step. Hence $\tau_{v_0,v_m}(t)\in S$, and since
		$v_0=u$ and $v_m=v$, it follows that $\tau_{u,v}(t)\in S$.
	\end{proof}
	
	We now apply the preceding lemma to the graph whose vertices are the standard matrix units and whose edges come from common $2\times2$ corners. Since this graph is connected, the local corner transvections propagate to all pairs of basis vectors.
	
	\begin{proposition}\label{prop:all-standard-transvections}
		Let $\cE=\{E_{ij}:1\le i,j\le n\}$ be the standard basis of $\Av=M_n(\K)$. Then for every distinct $u,v\in\cE$ and every $t\in\K$,
		\[
		\tau_{u,v}(t)\in\JMS(\cA).
		\]
	\end{proposition}
	
	\begin{proof}
		Let $\Gamma$ be the graph on $\cE$ in which two distinct vertices are joined exactly when they lie in a common set
		\[
		\{E_{ii},E_{ij},E_{ji},E_{jj}\}
		\qquad (1\le i<j\le n).
		\]
		By Proposition~\ref{prop:local-SL4}, both oriented transvections along every edge belong to $\JMS(\cA)$.
		
		The graph $\Gamma$ is connected. Indeed, if $i\ne j$, then the off-diagonal vertex
		$E_{ij}$ is adjacent to both $E_{ii}$ and $E_{jj}$, since these four basis
		vectors lie in the same set $\{E_{ii},E_{ij},E_{ji},E_{jj}\}$. Moreover, for each $i\ne 1$, the diagonal vertices $E_{11}$ and $E_{ii}$ are
		adjacent, since they lie in $\{E_{11},E_{1i},E_{i1},E_{ii}\}$. Thus every vertex of $\Gamma$ is connected to $E_{11}$, so $\Gamma$ is
		connected. Lemma~\ref{lem:path} now gives the claim.
	\end{proof}
	
	This gives all elementary transvections in the standard basis of $\Av$. Theorem~\ref{thm:SL-main} now follows immediately from Lemma~\ref{lem:gauss}.
	
	\begin{proof}[Proof of Theorem~\ref{thm:SL-main}]
		Order the standard basis $\cE$ as $(e_1,\dots,e_d)$ with $d=n^2$. Proposition~\ref{prop:all-standard-transvections} shows that $\JMS(\cA)$ contains every elementary transvection relative to this basis, so Lemma~\ref{lem:gauss} gives $\SL(\Av)\subseteq\JMS(\cA)$.
	\end{proof}
	
	\section{Determinants and abstract semigroup reductions}\label{sec:det-semigroup}
	
	In this section we complete the proof of the main theorem. By Theorem~\ref{thm:SL-main}, the semigroup $\JMS(\cA)$ already contains $\SL(\Av)$. To reach all of $\End_{\K}(\Av)$, it therefore remains to do two things: first, pass from $\SL(\Av)$ to $\GL(\Av)$ by proving that all determinant values occur on invertible elements of the semigroup; second, use one singular operator of rank $n^2-1$ together with the invertible part to recover every singular endomorphism.
	
	We begin by recording the determinant of a single Jordan multiplication operator. This formula will be used repeatedly in what follows.
	\begin{proposition}\label{prop:det-formula}
		Let $\K$ be a field of characteristic different from $2$, let $A\in M_n(\K)$, and let $\lambda_1,\dots,\lambda_n$ be the eigenvalues of $A$ in an algebraic closure of $\K$, counted with algebraic multiplicity. Then
		\begin{equation}\label{eq:detOpLA}
			\det(\opL_A)=2^{-n^2}\prod_{i,j=1}^n(\lambda_i+\lambda_j).
		\end{equation}
		Equivalently,
		\[
		\det(\opL_A)=2^{-n(n-1)}\det(A)\prod_{1\le i<j\le n}(\lambda_i+\lambda_j)^2.
		\]
		In particular, $\opL_A$ is invertible if and only if $\lambda_i+\lambda_j\ne 0$ for all $i,j$.
	\end{proposition}
	
	\begin{proof}
		Passing to an algebraic closure does not change the determinant of the underlying $\K$-linear operator, so it is enough to prove the formula after extending scalars and assuming that $\K$ is algebraically closed. Choose $S\in\GL_n(\K)$ such that $B:=S^{-1}AS$ is upper-triangular with diagonal entries $\lambda_1,\dots,\lambda_n$. Conjugation by $S$ induces an automorphism
		\[
		\Phi_S(X):=SXS^{-1}
		\qquad (X\in M_n(\K)),
		\]
		and $\Phi_S\opL_B=\opL_A\Phi_S$. Thus $\opL_A$ and $\opL_B$ are similar.
		
		Order the matrix units $E_{ij}$ by increasing $j$ and, for fixed $j$, by decreasing $i$. Since $B$ is upper-triangular,
		\[
		BE_{ij}\in\spanop\{E_{kj}:k\le i\},
		\qquad
		E_{ij}B\in\spanop\{E_{i\ell}:\ell\ge j\}.
		\]
		With the chosen order, every basis vector occurring in $BE_{ij}-\lambda_iE_{ij}$ or $E_{ij}B-\lambda_jE_{ij}$ comes after $E_{ij}$. Hence
		\[
		2\opL_B(E_{ij})=(\lambda_i+\lambda_j)E_{ij}
		+\text{(a linear combination of basis vectors occurring after $E_{ij}$)}.
		\]
		Therefore the matrix of $2\opL_B$ in this basis is lower triangular, with
		diagonal entry $\lambda_i+\lambda_j$ in the row corresponding to $E_{ij}$. Hence
		\[
		\det(2\opL_B)=\prod_{i,j=1}^n(\lambda_i+\lambda_j),
		\]
		which proves the first formula. The second follows by grouping the factors with $i=j$ and $i<j$. The invertibility criterion follows immediately from \eqref{eq:detOpLA}.
	\end{proof}
	
	\begin{remark}
		Equivalently, $2\opL_A=\Lambda_A+\Pi_A$ may be viewed as a Kronecker-sum
		operator. Under the identification $M_n(\K)\cong \K^n\otimes(\K^n)^*$, left multiplication by $A$ corresponds to $A\otimes I_n$, while right
		multiplication by $A$ corresponds to $I_n\otimes A^*$, where $A^*$ is the dual
		map of $A$ (equivalently, $A^t$ in dual bases). Thus
		Proposition~\ref{prop:det-formula} is precisely the resulting product formula
		for the eigenvalues of $2\opL_A$.
	\end{remark}
	
	For a monoid $S$, write $S^\times$ for its group of invertible elements. If
	$S\subseteq\End_{\K}(W)$ with $W$ finite-dimensional, define
	\[
	\Delta(S):=\{\det(T):T\in S^\times\}\le\K^\times.
	\]
	
	Once determinant values in the unit group are understood, the passage from $\SL(W)$ to $\GL(W)$ is immediate. The next proposition isolates this simple reduction.
	
	\begin{proposition}\label{prop:det-reduction}
		Let $W$ be a finite-dimensional $\K$-vector space, and let $S\subseteq\End_{\K}(W)$ be a monoid. Assume that
		\begin{enumerate}[label={\rm(\roman*)}]
			\item $\SL(W)\subseteq S$;
			\item $\Delta(S)=\K^\times$.
		\end{enumerate}
		Then $\GL(W)\subseteq S$.
	\end{proposition}
	
	\begin{proof}
		Let $G\in\GL(W)$. By \rm(ii), there exists $H\in S^\times$ with $\det(H)=\det(G)$. Then
		$\det(H^{-1}G)=1$, so $H^{-1}G\in\SL(W)\subseteq S$ by \rm(i). Since $H\in S$, it follows that $G=H(H^{-1}G)\in S$, as desired. 
	\end{proof}
	
	After the invertible case, the remaining issue is singular operators. The following lemma shows that, once all invertible operators are present, a single corank-one element is enough to generate the whole endomorphism semigroup.
	\begin{lemma}\label{lem:semigroup}
		Let $\K$ be a field, let $d\ge 1$, and let $S\subseteq M_d(\K)$ be a semigroup. Assume that
		\begin{enumerate}[label={\rm(\roman*)}]
			\item $\GL_d(\K)\subseteq S$;
			\item $S$ contains a singular element of rank $d-1$.
		\end{enumerate}
		Then $S=M_d(\K)$.
	\end{lemma}
	
	\begin{proof}
		If $d=1$, then $\GL_1(\K)=\K^\times\subseteq S$, and by \rm(ii) the semigroup
		also contains $0$, so $S=M_1(\K)$. Assume henceforth that $d\ge 2$.
		
		Let $A\in S$ be singular of rank $d-1$. There exist $G,H\in\GL_d(\K)$ such that
		\[
		GAH=\diag(I_{d-1},0)=:E.
		\]
		Hence $E\in S$.
		
		For each $k\in\{1,\dots,d\}$, let $P_k$ be the permutation matrix swapping the
		$k$th basis vector with the $d$th and fixing the others. Then
		$E_k:=P_kEP_k^{-1}\in S$, where $E_k$ is the diagonal idempotent with a single
		$0$ in position $k$. Since the $E_k$ commute, the semigroup contains all their
		products, and in particular every diagonal idempotent $\diag(I_r,0)$ with
		$0\le r\le d$.
		
		Finally, every matrix of rank $r$ is $\GL_d(\K)$-equivalent to
		$\diag(I_r,0)$, that is, every matrix of rank $r$ is of the form
		\[
		U\,\diag(I_r,0)\,W
		\]
		for suitable $U,W\in\GL_d(\K)$. Thus every matrix belongs to $S$.
	\end{proof}
	
	It therefore remains to exhibit one singular element of rank $n^2-1$ inside $\JMS(\cA)$.
	
	\begin{proposition}\label{prop:rank-d-1}
		Let $\K$ be a field of characteristic different from $2$. Then $\JMS(M_n(\K)^+)$ contains a singular operator of rank $n^2-1$.
	\end{proposition}
	
	\begin{proof}
		Take $A:=\diag(0,1,\dots,1)\in M_n(\K)$. On the standard basis $E_{ij}$ of $\Av$,
		\[
		\opL_A(E_{ij})=\frac{\alpha_i+\alpha_j}{2}E_{ij},
		\]
		where $\alpha_1=0$ and $\alpha_i=1$ for $i\ge 2$. Since $\operatorname{char}\K\ne 2$, the scalar $(\alpha_i+\alpha_j)/2$ vanishes only when $(i,j)=(1,1)$. Hence $\ker(\opL_A)=\K E_{11}$ and $\rk(\opL_A)=n^2-1$.
	\end{proof}
	
	\subsection{Infinite fields}\label{sec:infinite}
	We now prove determinant surjectivity over infinite fields. The argument uses a simple family of diagonal Jordan multipliers. For all but finitely many parameters, these operators are invertible and their inverses still lie in the semigroup; this gives enough flexibility to realize every element of $\K^\times$ as a determinant.
	
	In this subsection $\K$ is an infinite field of characteristic different from $2$.
	The next scalar identity is the only special calculation needed in the infinite-field case.
	\begin{lemma}\label{lem:scalar-4-factor}
		Let $\sigma\in\K$ satisfy $\sigma\notin\{0,2,-2,-4\}$, and define
		\[
		v_1:=-\frac{8}{\sigma(\sigma+2)},
		\qquad
		v_2:=\frac{\sigma+2}{\sigma-2},
		\qquad
		v_3:=\frac{(\sigma-2)(\sigma+4)}{8},
		\qquad
		v_4:=-\frac{\sigma}{\sigma+4}.
		\]
		Then
		\[
		v_1v_2v_3v_4=1,
		\qquad
		(1+v_1)(1+v_2)(1+v_3)(1+v_4)=\sigma.
		\]
	\end{lemma}
	
	\begin{proof}
		Both identities are verified by direct calculation.
	\end{proof}
	
	For $u\in\K$, let
	\[
	A(u):=\diag(u,1,\dots,1)\in M_n(\K),
	\qquad
	m(u):=\opL_{A(u)}\in\JMS(\cA).
	\]
	
	The point of Lemma~\ref{lem:scalar-4-factor} is to choose
	$v_1,\dots,v_4$ so that the composition of the four factors
	\[
	\opL_{A(v_k)},\qquad k=1,\dots,4,
	\]
	leaves the $E_{11}$-eigenvalue unchanged, since
	$v_1v_2v_3v_4=1$, while it adjusts the mixed eigenvalues on
	$E_{1j}$ and $E_{j1}$ by
	\[
	\prod_{k=1}^4\frac{1+v_k}{2}
	=
	\frac{4u}{(u+1)^2}.
	\]
	In this way $
	\opL_{A(u^{-1})}\opL_{A(v_1)}\opL_{A(v_2)}
	\opL_{A(v_3)}\opL_{A(v_4)}
	$
	acts on each standard basis vector in the same way as $m(u)^{-1}$, and
	therefore equals $m(u)^{-1}$.
	
	We now make this construction precise and record the finite exceptional set.
	
	\begin{proposition}\label{prop:cofinite-diagonal-inverses}
		There exists a finite subset $B\subseteq\K$ such that, for every $u\in\K\setminus B$, the operator $m(u)$ is invertible and
		\[
		m(u)^{-1}\in\JMS(\cA).
		\]
		One may take
		\[
		B:=\left\{
		u\in\K:
		u=0,\ \ u=-1,\ \ \text{or}\ \ \frac{64u}{(u+1)^2}\in\{2,-2,-4\}
		\right\}.
		\]
	\end{proposition}
	
	\begin{proof}
		On the standard basis $E_{ij}$ of $\Av$,
		\begin{equation}\label{eq:mu-eigen}
			m(u)(E_{11})=uE_{11},
			\qquad
			m(u)(E_{1j})=\frac{u+1}{2}E_{1j},
			\qquad
			m(u)(E_{j1})=\frac{u+1}{2}E_{j1}\quad (j>1),
		\end{equation}
		and $m(u)$ fixes each $E_{ij}$ with $i,j>1$. By Proposition~\ref{prop:det-formula}, $m(u)$ is invertible whenever $u\ne 0,-1$. Set
		\[
		\sigma(u):=\frac{64u}{(u+1)^2}.
		\]
		Assume now that $u\ne 0,-1$ and $\sigma(u)\notin\{2,-2,-4\}$. Choose $v_1,\dots,v_4 \in \K$ as in Lemma~\ref{lem:scalar-4-factor} for the scalar $\sigma(u)$, and put
		\[
		D_0:=\diag(u^{-1},1,\dots,1),
		\qquad
		D_j:=\diag(v_j,1,\dots,1)\quad (1\le j\le 4).
		\]
		Define
		\[
		T:=\opL_{D_0}\opL_{D_1}\opL_{D_2}\opL_{D_3}\opL_{D_4}\in\JMS(\cA).
		\]
		By \eqref{eq:mu-eigen}, the operator $T$ acts on $E_{11}$ by the scalar
		\[
		u^{-1}v_1v_2v_3v_4=u^{-1}.
		\]
		For $j>1$, the operator $T$ acts on $E_{1j}$ and $E_{j1}$ by
		\[
		\frac{u^{-1}+1}{2}\prod_{k=1}^4\frac{1+v_k}{2}
		=
		\frac{u+1}{2u}\cdot\frac{\sigma(u)}{16}
		=
		\frac{u+1}{2u}\cdot\frac{4u}{(u+1)^2}
		=
		\frac{2}{u+1}.
		\]
		Finally, $T$ fixes every $E_{ij}$ with $i,j>1$. These are exactly the eigenvalues of $m(u)^{-1}$, so $T=m(u)^{-1}$.
		
		For each fixed $c\in\{2,-2,-4\}$, the condition
		\[
		\frac{64u}{(u+1)^2}=c
		\]
		is equivalent to a quadratic equation in $u$, hence has at most two solutions. Therefore $B$ is finite.
	\end{proof}
	
	This is the key input for determinant surjectivity over infinite fields.
	\begin{theorem}\label{thm:delta-infinite}
		Let $\K$ be an infinite field of characteristic different from $2$, and let $n\ge 2$. Then
		\[
		\Delta(\JMS(M_n(\K)^+))=\K^\times.
		\]
	\end{theorem}
	
	\begin{remark}\label{rem:algclosed-simplification}
		If $\K$ is algebraically closed, the conclusion of Theorem~\ref{thm:delta-infinite} is immediate. Indeed, given $\gamma\in\K^\times$, choose $\mu\in\K^\times$ with $\mu^{n^2}=\gamma$. Then $\opL_{\mu I_n}=\mu\id_{\Av}$ belongs to $\JMS(\cA)$ and satisfies
		\[
		\det(\opL_{\mu I_n})=\det(\mu\id_{\Av})=\mu^{n^2}=\gamma.
		\]
		Thus $\gamma\in\Delta(\JMS(\cA))$. The content of Theorem~\ref{thm:delta-infinite} is that determinant surjectivity still holds over an arbitrary infinite field of characteristic different from $2$.
	\end{remark}
	
	\begin{proof}[Proof of Theorem~\ref{thm:delta-infinite}]
		Set $S:=\JMS(\cA)$ and $r:=2n-2$. Fix $\gamma\in\K^\times$. We shall produce $h_\gamma\in S^\times$ with $\det(h_\gamma)=\gamma$.
		
		Let $B\subseteq\K$ be the finite exceptional set from Proposition~\ref{prop:cofinite-diagonal-inverses}. For $y\in\K^\times$ put
		\[
		x:=\gamma y^{-r},
		\qquad
		a:=\frac{x(y-1)}{x-y},
		\qquad
		b:=\frac{y-1}{x-y},
		\]
		whenever $y\neq 1$ and $x\neq y$. Then
		\[
		\frac ab=x,
		\qquad
		\frac{a+1}{b+1}=y.
		\]
		We claim that $y$ may be chosen so that
		\[
		y\neq 1,\qquad x\neq 1,\qquad x\neq y,\qquad a,b\notin B.
		\]
		Indeed, the conditions $x=1$ and $x=y$ are equivalent to
		\[
		y^r=\gamma,\qquad y^{r+1}=\gamma,
		\]
		respectively, so together with $y=1$ they exclude only finitely many $y\in\K^\times$.
		
		Now fix $e\in B$. The condition $a=e$ is equivalent to
		\[
		x(y-1)=e(x-y),
		\]
		and substituting $x=\gamma y^{-r}$ yields
		\[
		e\,y^{r+1}+\gamma y-\gamma(1+e)=0.
		\]
		Likewise, $b=e$ is equivalent to
		\[
		(1+e)y^{r+1}-y^r-e\gamma=0.
		\]
		Thus, for each $e\in B$, the conditions $a=e$ and $b=e$ result in polynomial equations of degree at most $r+1$, each excluding only finitely many values of $y$. Since $B$ is finite and $\K$ is infinite, there exists $y\in\K^\times$ satisfying all of the above conditions.
		
		For such a choice of $y$, we have $a,b\notin B$, so Proposition~\ref{prop:cofinite-diagonal-inverses} gives $m(a),\,m(b)\in S^\times$. Set
		\[
		h_\gamma:=m(a)m(b)^{-1}\in S^\times.
		\]
		By \eqref{eq:mu-eigen}, the operator $m(u)$ has eigenvalue $u$ on $E_{11}$, eigenvalue $(u+1)/2$ with multiplicity $r=2n-2$, and eigenvalue $1$ on the remaining basis vectors; hence
		\[
		\det(m(u))=u\left(\frac{u+1}{2}\right)^r.
		\]
		Therefore
		\[
		\det(h_\gamma)
		=
		\frac{\det(m(a))}{\det(m(b))}
		=
		\frac ab\left(\frac{a+1}{b+1}\right)^r
		=
		xy^r
		=
		\gamma.
		\]
		Thus $\gamma\in\Delta(S)$. Since $\gamma\in\K^\times$ was arbitrary, it follows that $\Delta(\JMS(M_n(\K)^+))=\K^\times$.
	\end{proof}
	
	\begin{remark}\label{rem:large-finite-by-avoidance}
		Although Theorem~\ref{thm:delta-infinite} is stated for infinite fields, the same parameter-avoidance argument also works over finite fields once the field is large enough. More precisely, let $\F_q$ be a finite field of odd characteristic and put $r:=2n-2$. If $B\subseteq\F_q$ is the exceptional set from Proposition~\ref{prop:cofinite-diagonal-inverses}, then the proof goes through whenever
		\[
		q-1>2(r+1)(|B|+1).
		\]
		Since $|B|\le 8$, a simple sufficient condition is
		\[
		q-1>18(2n-1).
		\]
		Under this hypothesis one already obtains
		\[
		\Delta(\JMS(M_n(\F_q)^+))=\F_q^\times
		\]
		without using character sums.
	\end{remark}
	
	At this point the infinite-field case is immediate: determinant surjectivity gives the full invertible group, and Proposition~\ref{prop:rank-d-1} supplies a corank-one singular operator.
	\begin{corollary}\label{cor:main-infinite}
		Let $\K$ be an infinite field of characteristic different from $2$, and let $n\ge 2$. Then
		\[
		\GL(M_n(\K))\subseteq\JMS(M_n(\K)^+)
		\qquad\text{and}\qquad
		\JMS(M_n(\K)^+)=\End_{\K}(M_n(\K)).
		\]
	\end{corollary}
	
	\begin{proof}
		Set $S:=\JMS(\cA)$. By Theorem~\ref{thm:SL-main} and Theorem~\ref{thm:delta-infinite}, Proposition~\ref{prop:det-reduction} gives $\GL(\Av)\subseteq S$. Proposition~\ref{prop:rank-d-1} and Lemma~\ref{lem:semigroup} then imply $S=\End_{\K}(\Av)$.
	\end{proof}
	
	\subsection{Finite fields}\label{sec:finite}
	We next prove determinant surjectivity over finite fields. In this case the
	problem reduces to showing that the determinants of invertible Jordan multipliers generate $\K^\times$; for $q\ge 5$ this follows from a Jacobi-sum
	estimate, while $\F_3$ is handled separately.
	
	In this subsection, $\K=\F_q$ is a finite field of odd characteristic, and, as before,
	\[
	\cA:=M_n(\K)^+,
	\qquad
	\Av:=M_n(\K).
	\]
	
	We first record a simple fact specific to the finite setting: in a finite monoid of endomorphisms, every invertible element of the ambient endomorphism ring is already a unit of the monoid.
	\begin{lemma}\label{lem:finite-unit-group}
		Let $M\subseteq\End_{\K}(\Av)$ be a finite monoid. Then
		\[
		M^\times=M\cap\GL(\Av).
		\]
	\end{lemma}
	
	\begin{proof}
		The inclusion $M^\times\subseteq M\cap\GL(\Av)$ is clear. Conversely, let $T\in M\cap\GL(\Av)$. Since $M$ is finite, there exist integers $r>s\ge 0$ such that $T^r=T^s$. As $T$ is invertible in $\End_{\K}(\Av)$, we may compose with $T^{-s}$ in $\End_{\K}(\Av)$ and obtain
		$T^{r-s}=\id_{\Av}$. Hence $T^{-1}=T^{r-s-1}\in M$, because $M$ is closed under multiplication and contains $\id_{\Av}$. Thus $T\in M^\times$.
	\end{proof}
	
	In particular, if $S$ is a finite submonoid of $\End_{\K}(V)$, then the
	set of nonzero determinant values $\Delta(S)$ is a subgroup of $\K^\times$.

	The field $\F_3$ is exceptional and is treated directly.
	\begin{lemma}\label{lem:q3-det-surj}
		Let $\K=\F_3$ and let $n\ge 2$. Then
		\[
		\Delta(\JMS(M_n(\K)^+))=\K^\times.
		\]
	\end{lemma}
	
	\begin{proof}
		Set $S:=\JMS(\cA)$ and let
		\[
		B:=\begin{bmatrix}0&1\\ 1&1\end{bmatrix}\in M_2(\K).
		\]
		Then $\det(B)=2$ and $\Tr(B)=1$.
		
		Assume first that $n=2$. By Proposition~\ref{prop:det-formula},
		\[
		\det(\opL_B)=\det(B)\left(\frac{\Tr(B)}{2}\right)^2=2,
		\]
		so $2\in \Delta(S)$, and therefore $\Delta(S)=\K^\times$ since $\det(\opL_{I_2})=1$.
		
		Now assume that $n\ge 3$, and set
		\[
		A:=\diag(B,I_{n-2})\in M_n(\K).
		\]
		Write $X\in\Av$ in block form
		\[
		X=\begin{bmatrix} P&Q\\ R&T\end{bmatrix},
		\]
		where
		\[
		P\in M_2(\K),\quad Q\in M_{2,n-2}(\K),\quad R\in M_{n-2,2}(\K),\quad T\in M_{n-2}(\K).
		\]
		Then
		\[
		\opL_A(X)=
		\begin{bmatrix}
			\opL_B(P) & \dfrac{B+I_2}{2}\,Q\\[2mm]
			R\,\dfrac{B+I_2}{2} & T
		\end{bmatrix}.
		\]
		Hence $\opL_A$ is block diagonal with respect to the decomposition
		\[
		\Av=M_2(\K)\oplus M_{2,n-2}(\K)\oplus M_{n-2,2}(\K)\oplus M_{n-2}(\K),
		\]
		and its determinant is the product of the determinants on these four summands. On $M_{2,n-2}(\K)$, left multiplication by $(B+I_2)/2$ acts independently on the $n-2$ columns, so its determinant is $\det((B+I_2)/2)^{n-2}$; similarly, on $M_{n-2,2}(\K)$ the determinant is again $\det((B+I_2)/2)^{n-2}$. Therefore
		\[
		\det(\opL_A)=\det(\opL_B)\det\!\left(\frac{B+I_2}{2}\right)^{2(n-2)}.
		\]
		A direct computation shows that $\det((B+I_2)/2)=1$ in $\K$, so $\det(\opL_A)=2$. Hence $2\in\Delta(S)$, and therefore $\Delta(S)=\K^\times$.
	\end{proof}
	
	We now prove determinant surjectivity for finite fields of odd characteristic. The case $\F_3$ has already been settled separately, so the remaining argument is a character-sum estimate for a suitable family of invertible Jordan multipliers.
	
	\begin{theorem}\label{thm:finite-det-surj}
		Let $\K=\F_q$ be a finite field of odd characteristic, and let $n\ge 2$. Then
		\[
		\Delta(\JMS(M_n(\K)^+))=\K^\times.
		\]
	\end{theorem}
	
	\begin{proof}
		If $q=3$, the claim is Lemma~\ref{lem:q3-det-surj}. We may therefore assume that $q\ge 5$. Set $S:=\JMS(\cA)$ and $m:=2n-2$. The argument is by contradiction: if $\Delta(S)$ were a proper subgroup of $\K^\times$, then some nontrivial multiplicative character would be trivial on it. Evaluating this character on the determinants of a suitable family of invertible Jordan multipliers leads to a Jacobi-sum estimate, which yields the required contradiction.
		
		Since $S$ is finite, Lemma~\ref{lem:finite-unit-group} implies that
		$\Delta(S)=\det(S^\times)$ is a subgroup of $\K^\times$.
		
		For $t\in\K^\times\setminus\{-1\}$, set
		\[
		A_t:=\diag(t,1,\dots,1).
		\]
		Then $\opL_{A_t}\in S$, and by Proposition~\ref{prop:det-formula},
		or directly from its action on the standard basis,
		\[
		\det(\opL_{A_t})
		=
		t\left(\frac{t+1}{2}\right)^{2n-2}.
		\]
		This determinant is nonzero for $t\ne 0,-1$. Hence
		\[
		t\left(\frac{t+1}{2}\right)^m\in\Delta(S),
		\qquad m=2n-2,
		\]
		for every $t\in\K^\times\setminus\{-1\}$.
		
		Suppose for contradiction that $\Delta(S)$ is a proper subgroup of $\K^\times$. Then the quotient
		$
		\K^\times/\Delta(S)
		$
		is a nontrivial finite abelian group. Choosing a nontrivial character of this
		quotient and composing with the quotient map, we obtain a nontrivial
		multiplicative character
		\[
		\chi:\K^\times\to\C^\times
		\]
		which is trivial on $\Delta(S)$. For each integer $r\in\Z$ define
		\[
		\chi^r(a):=\chi(a)^r \qquad (a\in\K^\times),
		\]
		and then extend $\chi^r$ to $\K$ by $(\chi^r)(0)=0$. Define
		\[
		\Sigma:=\sum_{t\in\K}\chi(t)\chi^m(t+1).
		\]
		For every $t\in\K^\times\setminus\{-1\}$,
		\[
		1=\chi\!\left(t\left(\frac{t+1}{2}\right)^m\right)=\chi(t)\chi^m(t+1)\chi(2)^{-m},
		\]
		so
		\[
		\chi(t)\chi^m(t+1)=\chi(2)^m.
		\]
		Since $\chi(0)=\chi^m(0)=0$, the summands at $t=0$ and $t=-1$ both vanish, and all remaining $q-2$ summands are equal to the constant $\chi(2)^m$. It follows that
		\[
		|\Sigma|=q-2.
		\]
		
		We now estimate $\Sigma$. If $\chi^m$ is the trivial character on $\K^\times$, then $\chi^m(t+1)=1$ for every $t\ne -1$, and therefore
		\[
		\Sigma=\sum_{t\ne -1}\chi(t)=-\chi(-1),
		\]
		because $\sum_{t\in\K}\chi(t)=0$ and $\chi(0)=0$. Hence $|\Sigma|=1$.
		
		Assume next that $\chi^m$ is not the trivial character on $\K^\times$. Given multiplicative characters $\alpha,\beta:\K^\times\to\C^\times$, define their \emph{Jacobi sum} by
		\[
		J(\alpha,\beta):=\sum_{x\in\K}\alpha(x)\beta(1-x),
		\]
		where $\alpha$ and $\beta$ are extended to $\K$ by setting their value at $0$ equal to $0$; see, e.g., \cite[Chapter~5, \S3]{LidlNiederreiter1997}. After the substitution $x=-t$, we get
		\[
		J(\chi,\chi^m)
		=
		\sum_{t\in\K}\chi(-t)\chi^m(1+t)
		=
		\chi(-1)\sum_{t\in\K}\chi(t)\chi^m(t+1)
		=
		\chi(-1)\Sigma.
		\]
		Since $\chi(-1)^2=\chi(1)=1$, this is equivalent to
		\[
		\Sigma=\chi(-1)J(\chi,\chi^m).
		\]
		If $\chi^{m+1}$ is the trivial character on $\K^\times$, then $\chi^m=\chi^{-1}$ and the standard identity
		\[
		J(\chi,\chi^{-1})=-\chi(-1)
		\]
		for nontrivial $\chi$ gives $|\Sigma|=1$. If $\chi^{m+1}$ is not the trivial character on $\K^\times$, then the standard Jacobi-sum formula gives
		\[
		|J(\chi,\chi^m)|=\sqrt q.
		\]
		See, for example, \cite[Chapter~5, \S3]{LidlNiederreiter1997} for these Jacobi-sum identities. Hence in every case $|\Sigma|\le\sqrt q$, contradicting $|\Sigma|=q-2$ because $q\ge 5$ implies $q-2>\sqrt q$.
	\end{proof}
	
	This completes the finite-field determinant argument, and hence the finite-field case of the main theorem.
	\begin{corollary}\label{cor:main-finite}
		Let $\K$ be a finite field of odd characteristic, and let $n\ge 2$. Then
		\[
		\GL(M_n(\K))\subseteq\JMS(M_n(\K)^+)
		\qquad\text{and}\qquad
		\JMS(M_n(\K)^+)=\End_{\K}(M_n(\K)).
		\]
	\end{corollary}
	
	\begin{proof}
		Set $S:=\JMS(\cA)$. By Theorem~\ref{thm:SL-main} and Theorem~\ref{thm:finite-det-surj}, Proposition~\ref{prop:det-reduction} gives $\GL(\Av)\subseteq S$. Proposition~\ref{prop:rank-d-1} and Lemma~\ref{lem:semigroup} then imply $S=\End_{\K}(\Av)$.
	\end{proof}
	
	\subsection{Proof of Theorem~\ref{thm:main-all-fields}}
	
	Having treated both the infinite-field and finite-field cases, we are now ready to prove our main result.
	\begin{proof}[Proof of Theorem~\ref{thm:main-all-fields}]
		If $n=1$, then $M_1(\K)^+=\K$ and $\opL_A(X)=AX$, so $\JMS(M_1(\K)^+)=\End_{\K}(\K)$.
		
		Assume now that $n\ge 2$. If $\K$ is infinite, the result is Corollary~\ref{cor:main-infinite}. If $\K$ is finite, then $\operatorname{char}\K\ne 2$ means that $\K$ has odd cardinality, and the result is Corollary~\ref{cor:main-finite}.
	\end{proof}
	
	As an immediate consequence of Theorem~\ref{thm:main-all-fields}, we obtain the following.
	\begin{corollary}\label{cor:unit-group}
		Let $\K$ be a field of characteristic different from $2$, and let $n\ge 1$. Then
		\[
		\JMS(M_n(\K)^+)^\times=\GL(M_n(\K)).
		\]
	\end{corollary}
	
	\section{Concluding remarks}
	
	\begin{remark}[Further directions]
		Several natural quantitative questions remain open. For example, one may ask
		for effective bounds on the word length needed to express a given endomorphism
		of $M_n(\K)$ as a product of Jordan multiplication operators. It would also
		be interesting to determine smaller or more canonical generating families for
		$\JMS(M_n(\K)^+)$, and to understand how the minimal word length depends on
		$n$ and on the field $\K$.
		
		More concretely, one may ask for smaller but still natural families of
		matrices $\mathcal{G}\subseteq M_n(\K)$ whose corresponding multiplication
		operators already generate the full endomorphism semigroup:
		\[
		\langle \opL_A:A\in\mathcal{G}\rangle=\End_{\K}(M_n(\K)).
		\]
		A second natural quantitative question concerns the minimal number of factors
		needed to represent a given operator in this way. Namely, given
		$T\in\End_{\K}(M_n(\K))$, our result shows that there exist matrices
		$A_1,\dots,A_k\in M_n(\K)$ such that
		\[
		T=\opL_{A_1}\cdots\opL_{A_k}.
		\]
		If $k(T)$ denotes the minimal number of factors in such a representation, then
		one may ask for effective upper bounds on $k(T)$, for structural
		descriptions of operators with small word length, and for the asymptotic
		behaviour of maximal or typical values of $k(T)$ as $n$ varies and the ground
		field changes.
		
		\smallskip
		
		A broader direction is to study analogous semigroups for other Jordan
		algebras, as well as for natural subspaces and related matrix structures. The
		present result suggests that Jordan multiplication operators can generate
		surprisingly large linear semigroups, but it remains unclear which structural
		features of $M_n(\K)^+$ are responsible for this phenomenon and how far it
		extends. We have already begun a systematic investigation of these questions,
		both for further classes of simple Jordan algebras and for analogous problems in
		other nonassociative settings, including simple Lie algebras. Several of
		these projects are currently under active development and will appear in
		subsequent work.
	\end{remark}
	
	\begin{remark}[Exact symbolic verification]
		The explicit identities recorded in
		Appendices~\ref{app:Q-factorization} and~\ref{app:M2-computations} were
		verified by exact symbolic computation in the accompanying supplementary
		Mathematica files. In particular, the identities in
		Appendix~\ref{app:Q-factorization} are verified exactly over $\Q$, while those
		in Appendix~\ref{app:M2-computations} are verified exactly with symbolic
		parameters $t,s$ and coefficients in $\Z[1/2]$. Thus the displayed equalities
		are established exactly, rather than by testing special values.
	\end{remark}
	
	\section*{Acknowledgments}
	OpenAI’s GPT-5.4 was used for language editing and in the exploratory stage of this work to search for explicit finite factorizations and local operator identities. The explicit calculations used in the paper were subsequently
	verified by exact symbolic computation, and the corresponding supplementary verification files accompany the paper.
	
	\smallskip

	This research was funded by the European Union NextGenerationEU through the National Recovery and Resilience Plan 2021--2026. Institutional grant of University of Zagreb Faculty of Science (IK IA 1.1.3. Impact4Math).
	
	M.K. was supported by the Croatian Science Foundation under the project no.\ IP-2022-10-5008 (TEBAG) and acknowledges support from the project “Implementation of cutting-edge research and its application as part of the Scientific Center of Excellence for Quantum and Complex Systems, and Representations of Lie Algebras”, Grant No.\ PK.1.1.10.0004, co-financed by the European Union through the European Regional Development Fund -- Competitiveness and Cohesion Programme 2021-2027. 
	
	\appendix
	
	\section{An explicit rational factorization of \texorpdfstring{$\id+\opU_{E_{12}}$}{id+U\_E12}}\label{app:Q-factorization}
	
	This appendix records the only characteristic-zero factorization used in the body of the paper. Its sole purpose is to provide one explicit realization of $\id_{\Av}+\opU_{E_{12}}$ inside the semigroup generated by Jordan multiplication operators.
	
	\begin{proof}[Proof of Lemma~\ref{lem:E12-over-Q}]
		For $1\le i\le 16$, let $B_i\in M_2(\Q)$ be given by
		\[
		\begin{aligned}
			B_1&:=\begin{bmatrix}1&\frac32\\[2pt]\frac12&1\end{bmatrix},&
			B_2&:=\begin{bmatrix}1&-\frac32\\[2pt]-\frac12&1\end{bmatrix},\\[6pt]
			B_3=B_{15}&:=\begin{bmatrix}1&3\\[2pt]1&1\end{bmatrix},&
			B_4=B_{16}&:=\begin{bmatrix}1&-3\\[2pt]-1&1\end{bmatrix},\\[6pt]
			B_5=B_{13}&:=\begin{bmatrix}1&1\\[2pt]3&1\end{bmatrix},&
			B_6=B_{14}&:=\begin{bmatrix}1&-1\\[2pt]-3&1\end{bmatrix},\\[6pt]
			B_7=B_9&:=\begin{bmatrix}1&0\\[2pt]1&1\end{bmatrix},&
			B_8=B_{10}&:=\begin{bmatrix}1&0\\[2pt]-1&1\end{bmatrix},\\[6pt]
			B_{11}&:=\begin{bmatrix}1&\frac12\\[2pt]\frac32&1\end{bmatrix},&
			B_{12}&:=\begin{bmatrix}1&-\frac12\\[2pt]-\frac32&1\end{bmatrix}.
		\end{aligned}
		\]
		Relative to the standard basis $(E_{11},E_{12},E_{21},E_{22})$ of $M_2(\Q)$, exact multiplication of the corresponding $4\times 4$ operator matrices yields
		\begin{equation}\label{eq:Q-E12-factorization}
			\opL_{B_1}\opL_{B_2}\cdots\opL_{B_{16}}=
			\begin{bmatrix}
				1&0&0&0\\
				0&1&1&0\\
				0&0&1&0\\
				0&0&0&1
			\end{bmatrix}
			=\id_{M_2(\Q)}+\opU_{E_{12}}.
		\end{equation}
		Indeed, $\opU_{E_{12}}$ annihilates $E_{11},E_{12},E_{22}$ and sends $E_{21}$ to $E_{12}$. Therefore $\id_{M_2(\Q)}+\opU_{E_{12}}$ fixes $E_{11},E_{12},E_{22}$ and sends $E_{21}$ to $E_{21}+E_{12}$. This proves the lemma for $n=2$.
		
		Assume now that $n\ge 3$. For $1\le i\le 16$, set
		\[
		\widehat B_i:=\diag(B_i,I_{n-2})\in M_n(\Q).
		\]
		From now on, $E_{12}$ denotes the $n\times n$ matrix unit with its
		nonzero entry in position $(1,2)$. Write $X\in\Av=M_n(\Q)$ in $2\times 2$ block form:
		\[
		X=\begin{bmatrix}P&Q\\ R&S\end{bmatrix},
		\qquad
		P\in M_2(\Q),\quad Q\in M_{2,n-2}(\Q),\quad R\in M_{n-2,2}(\Q),\quad S\in M_{n-2}(\Q).
		\]
		Then
		\begin{equation}\label{eq:block-action-B}
			\opL_{\widehat B_i}(X)=
			\begin{bmatrix}
				\opL_{B_i}(P) & \dfrac{B_i+I_2}{2}\,Q\\[2mm]
				R\,\dfrac{B_i+I_2}{2} & S
			\end{bmatrix}.
		\end{equation}
		Set
		\[
		C_i:=\frac{B_i+I_2}{2}\in M_2(\Q).
		\]
		A direct multiplication gives
		\[
		C_1C_2=\frac{13}{16}I_2,
		\quad
		C_3C_4=\frac14I_2,
		\quad
		C_5C_6=\frac14I_2,
		\quad
		C_7C_8=I_2,
		\quad
		C_9C_{10}=I_2,
		\]
		\[
		C_{11}C_{12}=\frac{13}{16}I_2,
		\quad
		C_{13}C_{14}=\frac14I_2,
		\quad
		C_{15}C_{16}=\frac14I_2,
		\]
		and therefore
		\begin{equation}\label{eq:Ci-product}
			C_1C_2\cdots C_{16} = C_{16}\cdots C_2C_1 = \frac{169}{65536}\,I_2.
		\end{equation}
		
		We now choose a scalar $\sigma_0\in\Q$ so that the four additional diagonal factors exactly cancel the scalar appearing in the off-diagonal blocks in \eqref{eq:Ci-product}. At this point we use Lemma~\ref{lem:scalar-4-factor}, which is a purely
		algebraic identity and applies over $\Q$. A convenient choice is
		\[
		\sigma_0:=\left(\frac{1024}{13}\right)^2,
		\]
		for which
		\[
		\frac{\sigma_0}{16}=\frac{65536}{169}.
		\]
		Note that $\sigma_0\notin\{0,2,-2,-4\}$, so Lemma~\ref{lem:scalar-4-factor} applies. It provides rational numbers $v_1,\dots,v_4$ such that
		\[
		v_1v_2v_3v_4=1,
		\qquad
		(1+v_1)(1+v_2)(1+v_3)(1+v_4)=\sigma_0.
		\]
		Equivalently,
		\[
		\prod_{j=1}^4\frac{1+v_j}{2}=\frac{\sigma_0}{16}=\frac{65536}{169}.
		\]
		For $1\le j\le 4$, set
		\[
		D_j:=\diag(v_jI_2,I_{n-2})\in M_n(\Q).
		\]
		Then
		\begin{equation}\label{eq:block-action-D}
			\opL_{D_j}(X)=
			\begin{bmatrix}
				v_jP & \dfrac{1+v_j}{2}\,Q\\[2mm]
				\dfrac{1+v_j}{2}\,R & S
			\end{bmatrix}.
		\end{equation}
		Define
		\[
		\Theta:=\opL_{D_1}\opL_{D_2}\opL_{D_3}\opL_{D_4}\,\opL_{\widehat B_1}\opL_{\widehat B_2}\cdots\opL_{\widehat B_{16}}.
		\]
		Using \eqref{eq:block-action-B}, \eqref{eq:Ci-product}, and \eqref{eq:block-action-D}, we obtain
		\[
		\Theta(X)=
		\begin{bmatrix}
			(v_1v_2v_3v_4)\,(\opL_{B_1}\cdots \opL_{B_{16}})(P)
			&
			\left(\prod_{j=1}^4\frac{1+v_j}{2}\right)\left(C_1\cdots C_{16}\right)Q
			\\[3mm]
			R\left(C_{16}\cdots C_1\right)\left(\prod_{j=1}^4\frac{1+v_j}{2}\right)
			&
			S
		\end{bmatrix}.
		\]
		By \eqref{eq:Ci-product}, both $C_1\cdots C_{16}$ and $C_{16}\cdots C_1$ equal the same scalar matrix $\frac{169}{65536}I_2$, so the lower-left block simplifies in the same way as the upper-right block. Substituting \eqref{eq:Q-E12-factorization} and \eqref{eq:Ci-product} yields
		\[
		\Theta(X)=
		\begin{bmatrix}
			P+E_{12}PE_{12} & Q\\
			R & S
		\end{bmatrix}
		=
		X+E_{12}XE_{12}
		=
		(\id_{\Av}+\opU_{E_{12}})(X).
		\]
		Thus $\Theta=\id_{\Av}+\opU_{E_{12}}$.
	\end{proof}
	
	\section{Explicit \texorpdfstring{$M_2$}{M2}-calculations}\label{app:M2-computations}
	
	This appendix records the exact local $M_2$ identities used in Proposition~\ref{prop:M2}. They are the only explicit corner calculations on which the generation of elementary transvections depends. The formulas were first identified computationally and then verified by exact symbolic computation over the polynomial ring $\Z[1/2,t,s]$. Since all coefficients lie in $\Z[1/2]$, the resulting identities remain valid after specialization to every field of characteristic different from $2$.
	
	All operators below are written relative to the ordered basis
	\[
	\cE=(B_1,B_2,B_3,B_4)
	=
	\left(E_{21}+\frac12I_2,\ E_{12},\ I_2+E_{12},\ E_{11}-E_{22}\right)
	\]
	of $M_2(\K)$.
	
	\begin{lemma}\label{lem:M2-formulas}
		With respect to the basis $\cE$, the following identities hold:
		\begin{align*}
			u_{E_{12}}(t) &= \id_{M_2(\K)}+t\varepsilon_{21},\\
			g_{E_{12}}u_{E_{21}}(t)g_{E_{12}}^{-1}
			&=
			\id_{M_2(\K)}+t\varepsilon_{12},\\
			g_{R_+}u_{E_{21}}(t)g_{R_+}^{-1}
			&=
			\id_{M_2(\K)}+t\Bigl(\frac12\varepsilon_{12}-\frac12\varepsilon_{13}+\varepsilon_{14}\Bigr),\\
			g_{R_-}u_{E_{21}}(t)g_{R_-}^{-1}
			&=
			\id_{M_2(\K)}+t\Bigl(\frac12\varepsilon_{12}-\frac12\varepsilon_{13}-\varepsilon_{14}\Bigr),\\
			g_{2E_{21}}u_{E_{12}}(s)g_{2E_{21}}^{-1}
			&=
			\id_{M_2(\K)}+2s\,\varepsilon_{32},\\
			g_Ru_S(t)g_R^{-1}
			&=
			\id_{M_2(\K)}+t(-\varepsilon_{12}+\varepsilon_{32}+\varepsilon_{42}),
		\end{align*}
		where
		\[
		R_+:=\begin{bmatrix}1&1\\ -1&-1\end{bmatrix},
		\qquad
		R_-:=\begin{bmatrix}-1&1\\ -1&1\end{bmatrix},
		\qquad
		R:=\begin{bmatrix}-2&-1\\ 4&2\end{bmatrix},
		\qquad
		S:=\begin{bmatrix}-1&-1\\ 1&1\end{bmatrix}.
		\]
	\end{lemma}
	
	\begin{proof}
		In each case the matrices $E_{12}$, $E_{21}$, $R_+$, $R_-$, $R$, $S$, and $2E_{21}$ that enter the formulas are rank one and square-zero, so the corresponding operators $u_N(\cdot)$ and $g_Ru_S(\cdot)g_R^{-1}$ are defined and belong to $\JMS(M_2(\K)^+)$ by the results proved in the body of the paper.
		
		Each identity was verified by exact symbolic computation of the corresponding operator matrix relative to the basis $\cE$. For completeness, we also record the action on the basis vectors $B_1,B_2,B_3,B_4$, from which the displayed matrices may be read off directly.
		
		For the first identity,
		\[
		\opU_{E_{12}}(B_1)=E_{12}B_1E_{12}=E_{12}=B_2,
		\qquad
		\opU_{E_{12}}(B_2)=\opU_{E_{12}}(B_3)=\opU_{E_{12}}(B_4)=0,
		\]
		so $u_{E_{12}}(t)$ sends $B_1$ to $B_1+tB_2$ and fixes $B_2,B_3,B_4$.
		
		For the remaining identities, one obtains:
		\begin{align*}
			g_{E_{12}}u_{E_{21}}(t)g_{E_{12}}^{-1}&:\quad
			B_1\mapsto B_1,
			\quad
			B_2\mapsto B_2+tB_1,
			\quad
			B_3\mapsto B_3,
			\quad
			B_4\mapsto B_4;\\
			g_{R_+}u_{E_{21}}(t)g_{R_+}^{-1}&:\quad
			B_1\mapsto B_1,
			\quad
			B_2\mapsto B_2+\tfrac t2 B_1,
			\quad
			B_3\mapsto B_3-\tfrac t2 B_1,
			\quad
			B_4\mapsto B_4+t B_1;\\
			g_{R_-}u_{E_{21}}(t)g_{R_-}^{-1}&:\quad
			B_1\mapsto B_1,
			\quad
			B_2\mapsto B_2+\tfrac t2 B_1,
			\quad
			B_3\mapsto B_3-\tfrac t2 B_1,
			\quad
			B_4\mapsto B_4-t B_1;\\
			g_{2E_{21}}u_{E_{12}}(s)g_{2E_{21}}^{-1}&:\quad
			B_1\mapsto B_1,
			\quad
			B_2\mapsto B_2+2s B_3,
			\quad
			B_3\mapsto B_3,
			\quad
			B_4\mapsto B_4;\\
			g_Ru_S(t)g_R^{-1}&:\quad
			B_1\mapsto B_1,
			\quad
			B_2\mapsto B_2-t B_1+t B_3+t B_4,
			\quad
			B_3\mapsto B_3,
			\quad
			B_4\mapsto B_4.
		\end{align*}
		Writing these images in the basis $\cE$ gives exactly the matrices displayed above. Here we again use Lemma~\ref{lem:gR-inverse} for the inverses of the $g_R$.
	\end{proof}

\end{document}